\documentclass[a4paper,11pt]{article}

\usepackage[english]{babel}
\usepackage[utf8]{inputenc}

\usepackage{amsmath,amssymb,amsthm,latexsym}
\usepackage[justification=centering]{caption}
\usepackage{url,enumitem}

\usepackage{tikz}
\usetikzlibrary{patterns}

\usepackage{color}

\newtheorem{theorem}{Theorem}[section]
\newtheorem*{theorem*}{Theorem}
\newtheorem{lemma}[theorem]{Lemma}
\newtheorem{proposition}[theorem]{Proposition}
\newtheorem{corollary}[theorem]{Corollary}

\theoremstyle{definition}

\newtheorem*{fact*}{Fact}

\setlist{itemsep=0.1pt}

\def\A{\mathcal{A}}

\def\G{\mathcal{G}}

\def\N{\mathbb{N}}
\def\Q{\mathbb{Q}}
\def\Z{\mathbb{Z}}

\def\Qt{\Q[t^{\pm1}]}

\def\e{\varepsilon}

\def\psib{\overline{\psi}}

\def\Al{\mathfrak{A}}
\def\bl{\mathfrak{b}}
\def\BlMod{(\Al,\bl)}
\def\BlModD{\BlMod^{\oplus 2}}
\def\BlModDD{(\Al\oplus\Al,\bl\oplus\bl)}
\def\BlModN{\BlMod^{\oplus N}}

\def\BlModL{\BlMod^{\oplus \ell}}

\def\BlModl{\BlMod^{\oplus \ell}}
\def\BlModT{\BlMod^{\oplus 3}}
\def\E{\mathcal{E}}

\def\AS{{\mathrm{AS}}}
\def\Aut {{\mathrm{Aut}}}

\def\EV{{\mathrm{EV}}}
\def\Hol{{\mathrm{Hol}}}
\def\bHol{{\overline{\Hol}}}
\def\IHX{{\mathrm{IHX}}}
\def\LD{{\mathrm{LD}}}
\def\LE{{\mathrm{LE}}}
\def\LV{{\mathrm{LV}}}

\def\OR{{\mathrm{OR}}}

\def\YY{{\textrm{YY}}}
\def\H{{\textrm{H}}}
\def\tinf{\textrm{\tiny$\leq$}}

\newcommand{\fract}[2]{\hbox{\leavevmode
  \kern.1em \raise .25ex \hbox{\the\scriptfont0 $#1$}\kern-.1em }\big/
  {\hbox{\kern-.15em \lower .5ex \hbox{\the\scriptfont0 $#2$}} }}

\newcommand{\CenterFormula}[1]{\centerline{$\displaystyle{#1}$}}

\newcommand{\SixLegsPattern}{\vcenter{\hbox{{
 \begin{tikzpicture} [scale=0.1]
 \begin{scope} 
  \draw (-3.4,-2) -- (0,0) -- (0,4) (0,0) -- (3.4,-2);
  \draw (0,4) node {$\scriptscriptstyle{\bullet}$};
  \draw (-3.4,-2) node {$\scriptscriptstyle{\bullet}$};
  \draw (3.4,-2) node {$\scriptscriptstyle{\bullet}$};
 \end{scope}
 \begin{scope} [xshift=10cm]
  \draw (-3.4,-2) -- (0,0) -- (0,4) (0,0) -- (3.4,-2);
  \draw (0,4) node {$\scriptscriptstyle{\bullet}$};
  \draw (-3.4,-2) node {$\scriptscriptstyle{\bullet}$};
  \draw (3.4,-2) node {$\scriptscriptstyle{\bullet}$};
 \end{scope}
 \end{tikzpicture}}}}}
 
\newcommand{\sixlegsSchema}{\vcenter{\hbox{
 \begin{tikzpicture} [scale=0.1]
  \draw (-3.4,-2) -- (0,0) -- (0,4) (0,0) -- (3.4,-2);
  \draw (0,4) node {$\scriptstyle{\bullet}$} node[left] {1};
  \draw (-3.4,-2) node {$\scriptstyle{\bullet}$} node[below] {2};
  \draw (3.4,-2) node {$\scriptstyle{\bullet}$} node[below] {3};
 \begin{scope} [xshift=11.5cm]
  \draw (-3.4,-2) -- (0,0) -- (0,4) (0,0) -- (3.4,-2);
  \draw (0,4) node {$\scriptstyle{\bullet}$} node[left] {4};
  \draw (-3.4,-2) node {$\scriptstyle{\bullet}$} node[below] {5};
  \draw (3.4,-2) node {$\scriptstyle{\bullet}$} node[below] {6};
 \end{scope}
 \end{tikzpicture}}}}
 
\newcommand{\sixlegsop}[6]{\vcenter{\hbox{
 \begin{tikzpicture} [scale=0.15]
  \draw (-3.4,-2) -- (0,0) -- (0,4) (0,0) -- (3.4,-2);
  \draw (0,4) node {$\scriptstyle{\bullet}$} node[left] {$#1$};
  \draw (-3.4,-2) node {$\scriptstyle{\bullet}$} node[below] {$#2$};
  \draw (3.4,-2) node {$\scriptstyle{\bullet}$} node[below] {$#3$};
 \begin{scope} [xshift=11.5cm]
  \draw (-3.4,-2) -- (0,0) -- (0,4) (0,0) -- (3.4,-2);
  \draw (0,4) node {$\scriptstyle{\bullet}$} node[left] {$#4$};
  \draw (-3.4,-2) node {$\scriptstyle{\bullet}$} node[below] {$#5$};
  \draw (3.4,-2) node {$\scriptstyle{\bullet}$} node[below] {$#6$};
 \end{scope}
 \end{tikzpicture}}}}
 
\newcommand{\sixlegsGR}[6]{\vcenter{\hbox{
 \begin{tikzpicture} [scale=0.3]
  \draw (-3.4,-2) -- (0,0) -- (0,4) (0,0) -- (3.4,-2);
  \draw (0,4) node {$\bullet$} node[left] {$#1$};
  \draw (-3.4,-2) node {$\bullet$} node[below] {$#2$};
  \draw (3.4,-2) node {$\bullet$} node[below] {$#3$};
 \begin{scope} [xshift=13cm]
  \draw (-3.4,-2) -- (0,0) -- (0,4) (0,0) -- (3.4,-2);
  \draw (0,4) node {$\bullet$} node[left] {$#4$};
  \draw (-3.4,-2) node {$\bullet$} node[below] {$#5$};
  \draw (3.4,-2) node {$\bullet$} node[below] {$#6$};
 \end{scope}
 \end{tikzpicture}}}}
 
\newcommand{\sixlegsGRR}[6]{\vcenter{\hbox{
 \begin{tikzpicture} [scale=0.3]
  \draw (-3.4,-2) -- (0,0) -- (0,4) (0,0) -- (3.4,-2);
  \draw (0,4) node {$\bullet$} node[left] {$#1$};
  \draw (-3.4,-2) node {$\bullet$} node[below] {$#2\hspace{.8cm}$};
  \draw (3.4,-2) node {$\bullet$} node[below] {$\hspace{.8cm}#3$};
 \begin{scope} [xshift=18cm]
  \draw (-3.4,-2) -- (0,0) -- (0,4) (0,0) -- (3.4,-2);
  \draw (0,4) node {$\bullet$} node[left] {$#4$};
  \draw (-3.4,-2) node {$\bullet$} node[below] {$#5\hspace{.8cm}$};
  \draw (3.4,-2) node {$\bullet$} node[below] {$\hspace{.8cm}#6$};
 \end{scope}
 \end{tikzpicture}}}}

\newcommand{\sixlegs}[6]{\sixlegsop{#1\gamma_1}{#2\gamma_2}{#3\gamma_3}{#4\gamma_1}{#5\gamma_2}{#6\gamma_3}}

\newcommand{\sixlegsbis}[6]{\sixlegsop{#1\gamma_1}{#2\gamma_2}{#3\gamma_2}{#4\gamma_1}{#5\gamma_3}{#6\gamma_3}}

\newcommand{\FourLegsPattern}{\vcenter{\hbox{
 \begin{tikzpicture} [scale=0.1] 
  \draw (0,3.4) -- (2,0) -- (10,0) -- (12,3.4) (0,-3.4) -- (2,0) (10,0) -- (12,-3.4);
  \draw (0,3.4) node {$\scriptscriptstyle{\bullet}$};
  \draw (0,-3.4) node {$\scriptscriptstyle{\bullet}$};
  \draw (12,-3.4) node {$\scriptscriptstyle{\bullet}$};
  \draw (12,3.4) node {$\scriptscriptstyle{\bullet}$};
 \end{tikzpicture}}}}
 
\newcommand{\fourlegsop}[5]{\vcenter{\hbox{
 \begin{tikzpicture} [scale=#5] 
  \draw (0,3.4) -- (2,0) -- (10,0) -- (12,3.4) (0,-3.4) -- (2,0) (10,0) -- (12,-3.4);
  \draw (0,3.4) node {$\scriptstyle{\bullet}$} node[left] {$#1$};
  \draw (0,-3.4) node {$\scriptstyle{\bullet}$} node[left] {$#2$};
  \draw (12,-3.4) node {$\scriptstyle{\bullet}$} node[right] {$#3$};
  \draw (12,3.4) node {$\scriptstyle{\bullet}$} node[right] {$#4$};
 \end{tikzpicture}}}}
 
\newcommand{\fourlegsGR}[4]{\vcenter{\hbox{
 \begin{tikzpicture} [scale=0.2] 
  \draw (0,3.4) -- (2,0) -- (10,0) -- (12,3.4) (0,-3.4) -- (2,0) (10,0) -- (12,-3.4);
  \draw (0,3.4) node {$\bullet$} node[left] {$#1$};
  \draw (0,-3.4) node {$\bullet$} node[left] {$#2$};
  \draw (12,-3.4) node {$\bullet$} node[right] {$#3$};
  \draw (12,3.4) node {$\bullet$} node[right] {$#4$};
 \end{tikzpicture}}}}
 
\newcommand{\fourlegs}[4]{\fourlegsop{#1\gamma_1}{#2\gamma_2}{#3\gamma_1}{#4\gamma_2}{0.12}}

\newcommand{\fourlegsU}[4]{\fourlegsop{#1\gamma_1}{#2\gamma_1}{#3\gamma_1}{#4\gamma_1}{0.12}}
\newcommand{\fourlegsD}[4]{\fourlegsop{#1\gamma_2}{#2\gamma_2}{#3\gamma_2}{#4\gamma_2}{0.12}}

\newcommand{\fourlegsBas}{\vcenter{\hbox{\begin{tikzpicture} [scale=0.15]
\begin{scope}
  \draw (-3.4,-2) -- (0,0) -- (3.4,-2);
  \draw (-3.4,-2) node {$\scriptstyle{\bullet}$} node[below] {$\gamma_2$};
  \draw (3.4,-2) node {$\scriptstyle{\bullet}$} node[below] {$t\gamma_2$};
  \draw (0,0) .. controls +(0,6) and +(0,6) .. (11,0);
 \end{scope}
 \begin{scope} [xshift=11cm]
  \draw (-3.4,-2) -- (0,0) -- (3.4,-2);
  \draw (-3.4,-2) node {$\scriptstyle{\bullet}$} node[below] {$\gamma_3$};
  \draw (3.4,-2) node {$\scriptstyle{\bullet}$} node[below] {$t\gamma_3$};
 \end{scope}
 \end{tikzpicture}}}}

\newcommand{\fourlegsRevers}{\vcenter{\hbox{\begin{tikzpicture} [scale=0.15]
 \begin{scope} 
  \draw (-3.4,-2) -- (0,0) -- (0,4);
  \draw (0,4) node {$\scriptstyle{\bullet}$} node[left] {$\gamma_1$};
  \draw (-3.4,-2) node {$\scriptstyle{\bullet}$} node[below] {$\gamma_2$};
 \end{scope}
 \begin{scope} [xshift=8cm]
  \draw (-3.4,-2) -- (0,0) -- (0,4);
  \draw (0,4) node {$\scriptstyle{\bullet}$} node[left] {$\gamma_1$};
  \draw (-3.4,-2) node {$\scriptstyle{\bullet}$} node[above] {$\gamma_2$};
 \end{scope}
  \draw (0,0) .. controls +(5,-5) and +(4,-4) .. (8,0);
 \end{tikzpicture}}}}

\newcommand{\fourlegsLoop}{\vcenter{\hbox{\begin{tikzpicture} [scale=0.15]
\begin{scope}
  \draw (0,0) -- (0,4);
  \draw (0,4) node {$\scriptstyle{\bullet}$} node[left] {$\gamma_1$};
  \draw (0,-2) circle (2);
 \end{scope}
 \begin{scope} [xshift=8cm]
  \draw (-3.4,-2) -- (0,0) -- (0,4) (0,0) -- (3.4,-2);
  \draw (0,4) node {$\scriptstyle{\bullet}$} node[left] {$\gamma_1$};
  \draw (-3.4,-2) node {$\scriptstyle{\bullet}$} node[below] {$\gamma_3$};
  \draw (3.4,-2) node {$\scriptstyle{\bullet}$} node[below] {$t\gamma_3$};
 \end{scope}
 \end{tikzpicture}}}}

\newcommand{\twolegsop}[2]{\vcenter{\hbox{
 \begin{tikzpicture} [scale=0.3] 
  \draw (0,0) -- (2,0) .. controls +(1,1.7) and +(-1,1.7) .. (6,0) -- (8,0);
  \draw (2,0) .. controls +(1,-1.7) and +(-1,-1.7) .. (6,0);
  \draw (0,0) node {$\scriptstyle{\bullet}$} node[left] {$#1$};
  \draw (8,0) node {$\scriptstyle{\bullet}$} node[right] {$#2$};
 \end{tikzpicture}}}}

\newcommand{\twolegs}[2]{\twolegsop{#1\gamma_1}{#2\gamma_1}}

\newcommand{\zerolegs}{\vcenter{\hbox{
 \begin{tikzpicture} [scale=0.3] 
  \draw (0,0) .. controls +(1,1.7) and +(-1,1.7) .. (4,0);
  \draw (0,0) .. controls +(1,-1.7) and +(-1,-1.7) .. (4,0);
  \draw (0,0) -- (4,0);
 \end{tikzpicture}}}}

\newcommand{\GenLoop}{\vcenter{\hbox{
 \begin{tikzpicture} [scale=0.1]
  \draw (0,0) -- (0,4);
  \draw (0,4) node {$\scriptscriptstyle{\bullet}$};
  \draw (0,-2) circle (2);
 \end{tikzpicture}}}}

\newcommand{\GLoop}[3]{\vcenter{\hbox{
 \begin{tikzpicture} [scale=0.15]
  \draw (0,2) -- (0,4); \draw[->] (0,0) -- (0,2) node[right] {$#2$};
  \draw (0,4) node {$\scriptscriptstyle{\bullet}$} node[left] {$#1$};
  \draw (0,-2) circle (2); \draw[->] (0,-4) -- (0.1,-4) node[below] {$#3$};
 \end{tikzpicture}}}}

\newcommand{\PairingU}{\vcenter{\hbox{\begin{tikzpicture} [scale=0.4]
\draw (0,0) -- (0,4) (2.5,0) -- (2.5,4);
\draw (0,4) node {$\scriptstyle{\bullet}$} (2.5,4) node {$\scriptstyle{\bullet}$};
\draw (0,4) node[above right] {$v$} (2.5,4) node[above right] {$v'$};
\draw[->] (0,0) -- (0,2);
\draw[->] (2.5,0) -- (2.5,2);
\draw (0,2) node[right] {$P$} (2.5,2) node[right] {$Q$};
\draw (-0.43,-0.25) -- (0,0) -- (0.43,-0.25) (2.07,-0.25) -- (2.5,0) -- (2.93,-0.25);
\draw[dashed] (-0.86,-0.5) -- (-0.43,-0.25)  (0.86,-0.5) -- (0.43,-0.25)  (1.64,-0.5) -- (2.07,-0.25) (3.36,-0.5) -- (2.93,-0.25);
\end{tikzpicture}}}}
\newcommand{\PairingD}{\vcenter{\hbox{\begin{tikzpicture} [scale=0.4]
\draw (-0.43,-0.25) -- (0,0) -- (0.43,-0.25) (2.07,-0.25) -- (2.5,0) -- (2.93,-0.25);
\draw[dashed] (-0.86,-0.5) -- (-0.43,-0.25) (0.86,-0.5) -- (0.43,-0.25)  (1.64,-0.5) -- (2.07,-0.25) (3.36,-0.5) -- (2.93,-0.25);
\draw (0,0) -- (0,2.75) (2.5,0) -- (2.5,2.75) arc (0:180:1.25);
\draw[->] (1.2,4) -- (1.3,4);
\draw (1.25,4) node[above] {$P(t)Q(t^{-1})f_{vv'}(t)$};
\end{tikzpicture}}}}

\newcommand{\tiers}[1]{
  \draw[rotate=#1,color=white,line width=4pt] (0,0) -- (0,-2);
  \draw[rotate=#1] (0,0) -- (0,-2);}
\newcommand{\Orientation}{\vcenter{\hbox{\begin{tikzpicture} [scale=0.2]
  \draw (0,0) circle (1);
  \draw[<-] (-0.05,1) -- (0.05,1);
  \tiers{0}
  \tiers{120}
  \tiers{-120}
  \end{tikzpicture}}}}

\newcommand{\ASU}{\vcenter{\hbox{\begin{tikzpicture} [scale=0.3]
\draw (0,4) -- (2,2);
\draw (2,2) -- (4,4);
\draw (2,2) -- (2,0);
\end{tikzpicture}}}}
\newcommand{\ASD}{\vcenter{\hbox{\begin{tikzpicture} [scale=0.3]
\draw (8,2) .. controls +(2,0) and +(2.5,-1) .. (6,4);
\draw[white,line width=6pt] (8,2) .. controls +(-2,0) and +(-2.5,-1) .. (10,4);
\draw (8,2) .. controls +(-2,0) and +(-2.5,-1) .. (10,4);
\draw (8,0) -- (8,2);
\end{tikzpicture}}}}
\newcommand{\IHXU}{\vcenter{\hbox{\begin{tikzpicture} [scale=0.3]
\draw (18,4) -- (20,3) -- (20,1) -- (18,0);
\draw (20,1) -- (22,0);
\draw (20,3) -- (22,4);
\draw (20.5,2) node{1};
\end{tikzpicture}}}}
\newcommand{\IHXD}{\vcenter{\hbox{\begin{tikzpicture} [scale=0.3]
\draw (24,4) -- (25,2) -- (27,2) -- (28,4);
\draw (24,0) -- (25,2);
\draw (27,2) -- (28,0);
\draw (26,1.2) node{1};
\end{tikzpicture}}}}
\newcommand{\IHXT}{\vcenter{\hbox{\begin{tikzpicture} [scale=0.3]
\draw (30,4) -- (33,2) -- (34,0);
\draw[white,line width=6pt] (31,2) -- (34,4);
\draw (30,0) -- (31,2) -- (34,4);
\draw (31,2) -- (33,2);
\draw (32,1.2) node{1};
\end{tikzpicture}}}}
\newcommand{\LEU}{\vcenter{\hbox{\begin{tikzpicture} [scale=0.3]
\draw (9.4,0) -- (9.4,4);
\draw[->] (9.4,0) -- (9.4,3);
\draw (12.4,2.8) node{$xP+yQ$};
\end{tikzpicture}}}}
\newcommand{\LED}{\vcenter{\hbox{\begin{tikzpicture} [scale=0.3]
\draw (1,0) -- (1,4);
\draw[->] (1,0) -- (1,3);
\draw (1.8,2.8) node{$P$};
\end{tikzpicture}}}}
\newcommand{\LET}{\vcenter{\hbox{\begin{tikzpicture} [scale=0.3]
\draw (5.7,0) -- (5.7,4);
\draw[->] (5.7,0) -- (5.7,3);
\draw (6.5,2.8) node{$Q$};
\end{tikzpicture}}}}
\newcommand{\ORU}{\vcenter{\hbox{\begin{tikzpicture} [scale=0.3]
\draw (0,0) -- (0,4);
\draw[->] (0,0) -- (0,3);
\draw (1.5,2.8) node{$P(t)$};
\end{tikzpicture}}}}
\newcommand{\ORD}{\vcenter{\hbox{\begin{tikzpicture} [scale=0.3]
\draw (5,0) -- (5,4);
\draw[->] (5,4) -- (5,3);
\draw (7.2,2.8) node{$P(t^{-1})$};
\end{tikzpicture}}}}
\newcommand{\edge}[2]{
\draw[rotate=#1] (0,0) -- (0,3);
\draw[rotate=#1,->] (0,0) -- (0,1.5);
\draw[rotate=#1] (1,1.3) node{#2};}
\newcommand{\HolU}{\vcenter{\hbox{\begin{tikzpicture} [scale=0.3]
\edge{0}{$P$}
\edge{120}{$Q$}
\edge{240}{$R$}
\end{tikzpicture}}}}
\newcommand{\HolD}{\vcenter{\hbox{\begin{tikzpicture} [scale=0.3]
\edge{0}{$tP$}
\edge{120}{$tQ$}
\edge{240}{$tR$}
\end{tikzpicture}}}}
\newcommand{\LVU}{\vcenter{\hbox{\begin{tikzpicture} [scale=0.3]
\draw (13,0) -- (13,4);
\draw (13,1) node[right] {$D$};
\draw (13,4) node{$\bullet$};
\draw (13,4) node[right] {$x\gamma_1+y\gamma_2$};
\draw (13,4) node[below left] {$\scriptstyle{v}$};
\end{tikzpicture}}}}
\newcommand{\LVD}{\vcenter{\hbox{\begin{tikzpicture} [scale=0.3]
\draw (1,0) -- (1,4);
\draw (1,1) node[right] {$D_1$};
\draw (1,4) node{$\bullet$};
\draw (1,4) node[right] {$\gamma_1$};
\draw (1,4) node[below left] {$\scriptstyle{v}$};
\end{tikzpicture}}}}
\newcommand{\LVT}{\vcenter{\hbox{\begin{tikzpicture} [scale=0.3]
\draw (7.5,0) -- (7.5,4);
\draw (7.5,1) node[right] {$D_2$};
\draw (7.5,4) node{$\bullet$};
\draw (7.5,4) node[right] {$\gamma_2$};
\draw (7.5,4) node[below left] {$\scriptstyle{v}$};
\end{tikzpicture}}}}
\newcommand{\EVU}{\vcenter{\hbox{\begin{tikzpicture} [scale=0.3]
\draw (-0.8,0) -- (-0.8,4);
\draw (-0.8,4) node{$\bullet$};
\draw (-0.8,4) node[below left] {$\scriptstyle{v}$};
\draw[->] (-0.8,0) -- (-0.8,2);
\draw (-0.8,4) node[right] {$\gamma$};
\draw (0.6,1.8) node{$PQ$};
\draw (-0.8,0) node[right] {$D$};
\end{tikzpicture}}}}
\newcommand{\EVD}{\vcenter{\hbox{\begin{tikzpicture} [scale=0.3]
\draw (5,0) -- (5,4);
\draw (5,4) node{$\bullet$};
\draw (5,4) node[below left] {$\scriptstyle{v}$};
\draw[->] (5,0) -- (5,2);
\draw (5,4) node[right] {$Q\gamma$};
\draw (6,1.8) node{$P$};
\draw (5,0) node[right] {$D'$};
\end{tikzpicture}}}}
\newcommand{\LDU}{\vcenter{\hbox{\begin{tikzpicture} [scale=0.3]
\draw (-1,0) -- (-1,4);
\draw (-1,2) node[right] {$1$};
\draw (-1,4) node{$\bullet$};
\draw (-1,4) node[below left] {$\scriptstyle{v_1}$};
\draw (-1,4) node[right] {$\gamma_1$};
\draw (2.5,0) -- (2.5,4);
\draw (2.5,2) node[right] {$1$};
\draw (2.5,4) node{$\bullet$};
\draw (2.5,4) node[below left] {$\scriptstyle{v_2}$};
\draw (2.5,4) node[right] {$\gamma_2$};
\draw (2.5,0) node[right] {$D$};
\end{tikzpicture}}}}
\newcommand{\LDD}{\vcenter{\hbox{\begin{tikzpicture} [scale=0.3]
\draw (8,0) -- (8,4);
\draw (8,2) node[right] {$1$};
\draw (8,4) node{$\bullet$};
\draw (8,4) node[below left] {$\scriptstyle{v_1}$};
\draw (8,4) node[right] {$\gamma_1$};
\draw (11.5,0) -- (11.5,4);
\draw (11.5,2) node[right] {$1$};
\draw (11.5,4) node{$\bullet$};
\draw (11.5,4) node[below left] {$\scriptstyle{v_2}$};
\draw (11.5,4) node[right] {$\gamma_2$};
\draw (11.5,0) node[right] {$D'$};
\end{tikzpicture}}}}
\newcommand{\LDT}{\vcenter{\hbox{\begin{tikzpicture} [scale=0.3]
\draw (16,0) -- (16,2.75) (18.5,0) -- (18.5,2.75) arc (0:180:1.25);
\draw[->] (17.2,4) -- (17.3,4);
\draw (17.25,4) node[above] {$P$};
\draw (18.5,0) node[right] {$D''$};
\end{tikzpicture}}}}
\newcommand{\HolPU}{\vcenter{\hbox{\begin{tikzpicture} [scale=0.3]
\draw (0,0) -- (0,3);
\draw (0,-1) circle (1);
\draw[->] (-0.1,-2) -- (0,-2);
\draw (0,-3) node{$g$};
\draw[->] (0,0) -- (0,1.5);
\draw (0,1.3) node[right] {$f$};
\end{tikzpicture}}}}
\newcommand{\HolPD}{\vcenter{\hbox{\begin{tikzpicture} [scale=0.3]
\draw (0,0) -- (0,3);
\draw (0,-1) circle (1);
\draw[->] (-0.1,-2) -- (0,-2);
\draw (0,-3) node{$g$};
\draw[->] (0,0) -- (0,1.5);
\draw (0,1.3) node[right] {$tf$};
\end{tikzpicture}}}}

\newcommand{\Strut}{\vcenter{\hbox{\begin{tikzpicture} 
\draw (0,-0.5) -- (0,0);
\draw (0,-0.5) node{$\scriptstyle{\bullet}$};
\draw (0,0) node{$\scriptstyle{\bullet}$};
\end{tikzpicture}}}}

\definecolor{purple}{rgb}{0.63, 0.36, 0.94}

\title{Toward universality in degree 2 of the Kricker lift of the
  Kontsevich integral and the Lescop equivariant invariant}

\author{Benjamin Audoux $\&$ Delphine Moussard}

\date{}

\begin{document}

\maketitle

\begin{abstract}
In the setting of finite type invariants for null-homologous knots in rational homology 3--spheres with respect to 
null Lagrangian-preserving surgeries, there are two candidates to be universal invariants, defined
respectively by Kricker and Lescop. In a previous paper, the second
author defined maps between spaces of Jacobi diagrams. Injectivity for
these maps would imply that Kricker and Lescop invariants are indeed universal invariants; 
this would prove in particular that these two invariants are equivalent. 
In the present paper, we investigate the
injectivity status of these maps for degree 2 invariants, in the case
of knots whose Blanchfield modules are direct sums of isomorphic Blanchfield modules of \mbox{$\Q$--dimension}
two. We prove that they are always injective except in one case, for
which we determine explicitly the kernel.

\vspace{1ex}

\vspace{1ex}

\end{abstract}

\tableofcontents

\section{Introduction}

The work presented here has its source in the notion of finite type invariants. 
This notion first appeared in independent works of Goussarov and Vassiliev involving invariants of knots in the 3--dimensional 
sphere $S^3$; in this case, finite type invariants are also called Vassiliev invariants. 
Finite type invariants of knots in $S^3$ are defined by their polynomial behaviour with respect to crossing changes. 
The discovery of the Kontsevich integral, which is a universal invariant among all finite type invariants of knots in $S^3$, revealed the depth of this class 
of invariants. It is known, for instance, that it dominates all Witten-Reshetikhin-Turaev quantum invariants. 
A theory of finite type invariants can be defined for any kind of topological objects provided that an elementary move on the set of these objects is fixed; 
the finite type invariants are defined by their polynomial behaviour with respect to this elementary move. 
For 3--dimensional manifolds, the notion of finite type invariants was introduced by Ohtsuki \cite{Oht4}, who constructed the first examples for integral 
homology 3--spheres, and it has been widely developed and generalized since then. In particular, Goussarov and Habiro independently developed 
a theory which involves any 3--dimensional manifolds---and their knots---and which contains the Ohtsuki theory for $\Z$--spheres \cite{GGP,Hab}. 
In this case, the elementary move is the so-called Borromean surgery.

Garoufalidis and Rozansky introduced in \cite{GR} a theory of finite type invariants for knots in integral homology 3--spheres with respect to null-moves, 
which are Borromean surgeries satisfying a homological condition with respect to the knot. 
This theory was adapted to the ``rational homology setting'' by Lescop \cite{Les3} who defined a theory of finite type invariants for null-homologous knots 
in rational homology 3--spheres with respect to null Lagrangian-preserving surgeries. 
In these theories, the degree $0$ and $1$ invariants are well understood and, up to them, there are two candidates to be universal finite type invariants, namely 
the Kricker rational lift of the Kontsevich integral \cite{Kri,GK} and the Lescop equivariant invariant built from integrals over configuration spaces \cite{Les2}. 
Both of them are known to be universal finite type invariants in two situations already: for knots in integral homology 3--spheres with trivial Alexander 
polynomial, with respect to null-moves \cite{GR}, and for null-homologous knots in rational homology 3--spheres with trivial Alexander polynomial, with respect to null Lagrangian-preserving surgeries \cite{M7}. In particular, the Kricker invariant and the Lescop invariant are equivalent for such knots---in
the sense that they separate the same pairs of knots. Lescop conjectured in \cite{Les3} that this equivalence holds in general. 

Universal finite type invariants are known in other settings: the Kontsevich integral for links in $S^3$ \cite{BN}, the Le--Murakami--Ohtsuki invariant 
and the Kontsevich--Kuperberg--Thurston invariant for integral homology 3--spheres \cite{Le} and for rational homology 3--spheres \cite{M2}. 
To establish universality of these invariants, the general idea is to give a combinatorial description of the graded space associated with the theory 
by identifying it with a graded space of diagrams. Such a project is developed in \cite{M7} to study the universality of the Kricker and the Lescop invariants 
as finite type invariants of \emph{$\Q$SK--pairs}, which are pairs made of a rational homology 3--sphere and a null-homologous knot in it. 

Null Lagrangian-preserving surgeries preserve the Blanchfield module (defined over $\Q$), so one can reduce the study of finite type invariants 
of $\Q$SK--pairs to the set of $\Q$SK--pairs with a fixed Blanchfield module. In order to describe the graded space $\G\BlMod$ associated 
with finite type invariants of $\Q$SK--pairs with Blanchfield module $\BlMod$, a graded space of diagrams $\A^{aug}\BlMod$ is constructed in \cite{M7}, 
together with a surjective map $\varphi:\A^{aug}\BlMod\to\G\BlMod$. Injectivity of this map would imply universality of the Kricker invariant 
and the Lescop invariant for $\Q$SK--pairs with Blanchfield module $\BlMod$ and consequently equivalence of these two invariants for such $\Q$SK--pairs. 

Let $\BlMod$ be any Blanchfield module with annihilator $\delta\in\Qt$. As detailed in \cite{M7}, we can focus on the subspace $\G^b\BlMod=\oplus_{n\in\Z}\G^b_n\BlMod$ 
of $\G\BlMod$ associated with Borromean surgeries and study the restricted map $\varphi:\A\BlMod\to\G^b\BlMod$ defined on a subspace $\A\BlMod$ of $\A^{aug}\BlMod$. 
Both the Lescop and the Kricker invariants are families $Z=(Z_n)_{n\in\N}$ of finite type invariants, where $Z_n$ has degree $n$ when $n$ is even and 
$Z_n$ is trivial when $n$ is odd. For $\Q$SK--pairs with Blanchfield module $\BlMod$, $Z_n$ takes values in a space $\A_n(\delta)$ 
of trivalent graphs with edges labelled in $\frac{1}{\delta}\Qt$. The finiteness properties imply that $Z_n$ induces a map on $\G^b_n\BlMod$. 
The map $\varphi:\A\BlMod\to\G^b\BlMod$ decomposes as the direct sum of maps $\varphi_n:\A_n\BlMod\to\G^b_n\BlMod$. 
Composing with $Z_n$, we get a map $\psi_n:\A_n\BlMod\to\A_n(\delta)$; this provides the following commutative diagram: 
\[
 \vcenter{\hbox{\begin{tikzpicture} [xscale=1.4,yscale=0.8]
  \draw (2,2) node {$\G^b_n\BlMod$};
  \draw (2,-2) node {$\A_n(\delta)$};
  \draw (0,0) node {$\A_n\BlMod$};
  \draw[->>] (0.4,0.5) -- (1.6,1.5); \draw (0.9,1.3) node {$\varphi_n$};
  \draw[->] (0.4,-0.5) -- (1.6,-1.5); \draw (0.9,-1.3) node {$\psi_n$};
  \draw[->] (2,1.5) -- (2,-1.5); \draw (2.4,0) node {$Z_n$};
 \end{tikzpicture}}}.
\] 
Note that the injectivity of $\psi_n$ implies the injectivity of $\varphi_n$. When $\BlMod$ is a direct sum of $N$ isomorphic Blanchfield modules,
it has been  established in \cite{M7} that $\psi_n$ is an isomorphism when $n\leq\frac23N$. In particular, this applies for any $n\in\N$ 
when $\BlMod$ is the trivial Blanchfield module. 

In this paper, we look into the case $n=2$ when
$\BlMod$ is a direct sum of $N$ isomorphic Blanchfield
modules of $\Q$-dimension two. According to the above-mentioned result,
the map $\psi_2$ is then injective as soon as $N\geq 3$. The only
remaining cases are hence $N=1$ and $N=2$. We prove the following (Propositions \ref{proponecopy}, \ref{proptwocopies} and \ref{propt+1}):
\begin{theorem}
 If $\BlMod$ is a Blanchfield module of $\Q$-dimension two, with annihilator $\delta$, then:
 \begin{enumerate}
 \item the map $\psi_2:\A_2\BlMod\to\A_2(\delta)$ is injective but not surjective;
 \item the map $\psi_2:\A_2\BlModDD\to\A_2(\delta)$ is injective if and only if $\delta\neq t+1+t^{-1}$; in this case, it is an isomorphism.
 \end{enumerate}
\end{theorem}

It follows that, in degree 2, Kricker and Lescop invariants are
indeed universal and equivalent for $\Q$SK--pairs with a Blanchfield
module which is either of $\Q$--dimension two or a direct sum of isomorphic Blanchfield modules of $\Q$--dimension two, 
except in one exceptional case. 
But the most interesting, though unexpected, outcome of the above
theorem is this latter exceptional case---namely the case of a Blanchfield module which
is a direct sum of two isomorphic Blanchfield modules of order
$t+1+t^{-1}$---for which the map $\psi_2$ is \emph{not} injective.
The kernel of $\psi_2$ in this situation is explicited in Proposition
\ref{proptwocopies}. A topological realization $C$ is given in Figure \ref{fig:CounterExample}: 
$C$ is a linear combination of $\Q$SK-pairs whose class in $\G_2(\Al,\bl)$ is the image by $\varphi_2$ of a generator of the kernel of $\psi_2$. 
This leads to two alternatives. Either $C$ has topological reasons to vanish in $\G_2(\Al,\bl)$, then the map $\varphi_2$ itself is not injective and some more efforts should be done to understand the combinatorial nature of
$\G_n(\Al,\bl)$; or the Kricker and Lescop invariants do not induce,
in general, injective maps on $\G_n^b(\Al,\bl)$, suggesting the
existence of some yet unknown finite type invariants in this
setting. In both cases, the discussion is recentered on the explicited
counterexample which appears as a key example that should be studied further.

\begin{figure}
 \[
(M,K):=\vcenter{\hbox{\includegraphics{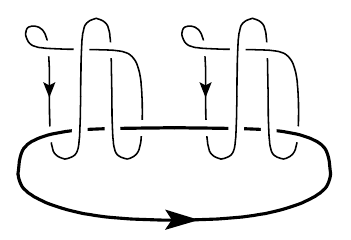}}}
\]
  \begin{eqnarray*}
C&:=&2\vcenter{\hbox{\includegraphics{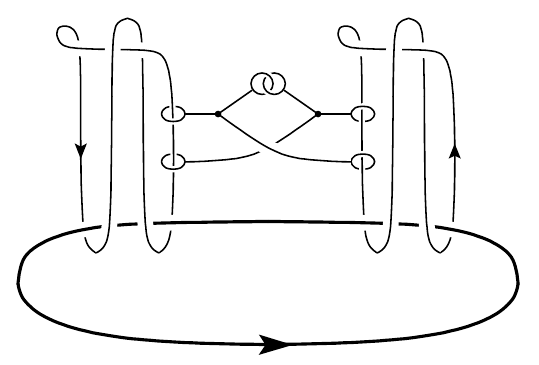}}}\ +\ \vcenter{\hbox{\includegraphics{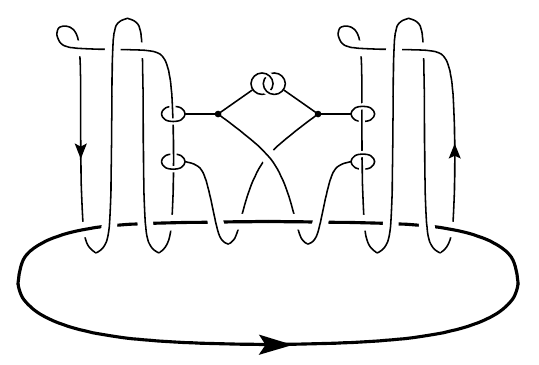}}}\\
    &&\hspace{2cm}-2\vcenter{\hbox{\includegraphics{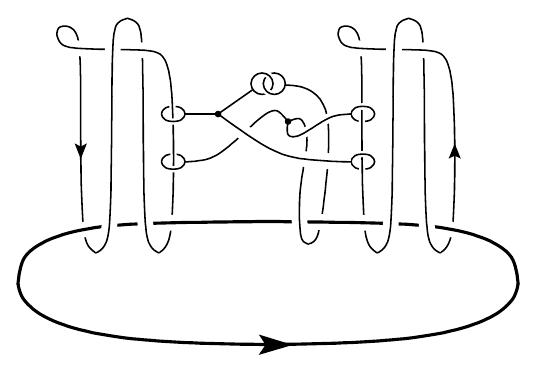}}}\ -\ \vcenter{\hbox{\includegraphics{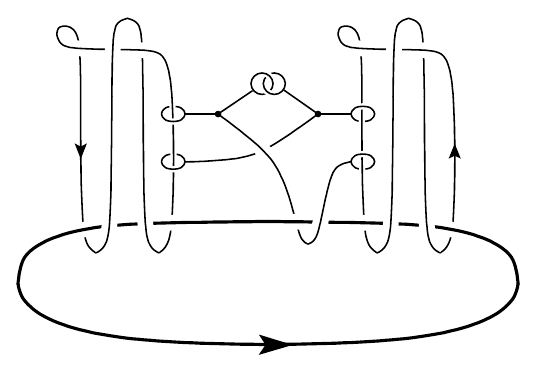}}}
  \end{eqnarray*}
\hspace{1cm}where $\vcenter{\hbox{\includegraphics{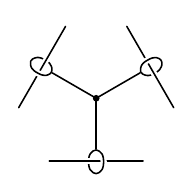}}}$ stands for $\vcenter{\hbox{\includegraphics{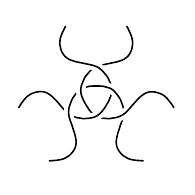}}}$

  \caption{A topological realization for a generator of the kernel of
    $\psi_2$\vspace{1ex}\\{\footnotesize Each picture represents the $\Q$SK--pair
      obtained by considering the copy of the thick unknot in the
      rational homology 3--sphere obtained by $0$--surgery on the
      other two knots. The sum corresponds to the image by
      $\varphi_2$ of the generator of $\textrm{Ker}(\psi_2)$ given in
      Proposition \ref{proptwocopies}.
      There is indeed a correspondence between the four $\H$--diagrams
      in the expression of this generator and the four terms in $C$,
      which are all of the form $(M,K)(T_1)(T_2)$ where $T_1$ and
      $T_2$ denote the two tripod graphs and $Y(T)$ denotes the result
      of the borromean surgery along $T$ on $Y$. More precisely, each
      $\H$--diagram is sent to 
      $(M,K)-(M,K)(T_1)-(M,K)(T_2)+(M,K)(T_1)(T_2)$, but $(M,K)(T_1)=(M,K)(T_2)=(M,K)$.
      See \cite{M1} for the computation of the Alexander module of $(M,K)$, \cite[Lemma 2.1]{GGP} 
      for the explicit action of the tripod graphs and \cite{M7} for other definitions and details.}}
  \label{fig:CounterExample}
\end{figure}

\paragraph{Acknowledgments.}
This work has been initiated while the second author was visiting the
first one in Aix--Marseille Université, courtesy of the ANR research project “VasKho” ANR-11-JS01-00201.
Hence, it has been carried out in the framework of Archimède Labex (ANR-11-LABX-0033) and of the A*MIDEX project (ANR-11-IDEX-0001-02), 
funded by the ``Investissements d'Avenir" French Government programme managed by the French National Research Agency (ANR).
The second author is supported by a Postdoctoral Fellowship of the Ja\-pan Society for the Promotion of Science. 
She is grateful to Tomotada Ohtsuki and the Research Institute for Mathematical Sciences for their support. 
While working on the contents of this article, she has also been supported by the Italian FIRB project ``Geometry and topology of low-dimensional
manifolds'', RBFR10GHHH, and by the University of Bourgogne.

\section{Definitions and strategy}

\subsection{Definitions}

\paragraph{Blanchfield modules.}
A \emph{Blanchfield module} is a pair $\BlMod$ such that:
\begin{enumerate}[label=(\roman*)]
 \item $\Al$ is a finitely generated torsion $\Qt$-module;
 \item\label{item:1-t} multiplication by $(1-t)$ defines an isomorphism of $\Al$;
 \item\label{item:Hermitian} $\bl:\Al\times\Al\to\fract{\Q(t)}{\Qt}$ is a non-degenerate hermitian form, {\it i.e.} 
  $\bl(\eta,\gamma)(t)=\bl(\gamma,\eta)(t^{-1})$, $\bl(P(t)\gamma,\eta)=P(t)\bl(\gamma,\eta)$, and if $\bl(\gamma,\eta)=0$ for all $\eta\in\Al$, then $\gamma=0$. 
\end{enumerate}

Since $\Qt$ is a principal ideal domain, there is a well-defined (up to multiplication by an invertible element of $\Qt$) annihilator $\delta\in\Qt$ for $\Al$. 
Condition \ref{item:1-t} implies that $\delta(1)\neq 0$ and Condition
\ref{item:Hermitian} that $\delta$ is symmetric, {\it i.e.} $\delta(t^{-1})=\upsilon(t)\delta(t)$ 
with $\upsilon(t)$ invertible in $\Qt$; see \cite[Section 3.2]{M1} for
more details. Moreover, it follows from $\bl$ being hermitian that
$\bl(\gamma,\eta)\in \frac{1}{P}\Qt$ if $\gamma$ has order $P$. 

In this paper, we focus on Blanchfield modules of $\Q$--dimension
$2$. In this case, either $\Al$ is cyclic, or it is a direct sum of two $\Qt$--modules 
with the same order. In this latter case, it follows from $\delta$
being symmetric and $\delta(1)\neq0$ that $\delta(t)=t+1$.

\paragraph{Spaces of $\BlMod$--colored diagrams.}

Fix a Blanchfield module $\BlMod$ and let $\delta\in\Qt$ be the annihilator of $\Al$.
An \emph{$\BlMod$--colored diagram} $D$ is a uni-trivalent graph without
strut ($\Strut$), given with:
\begin{itemize}
 \item an orientation for each trivalent vertex, that is a cyclic order 
  of the three half-edges that meet at this vertex; 
 \item an orientation and a label in $\Qt$ for each edge;
 \item a label in $\Al$ for each univalent vertex;
 \item a rational fraction $f_{vv'}^D(t)\in\frac{1}{\delta}{\Qt}$ for
  each pair $(v,v')$ of distinct univalent vertices of $D$, satisfying
  $f_{v'v}^D(t)=f_{vv'}^D(t^{-1})$ and $f_{vv'}^D(t)\ mod\ \Qt=\bl(\gamma_v,\gamma_{v'})$, where
  $\gamma_v$ and $\gamma_{v'}$ are the labels of $v$ and $v'$ respectively.   
\end{itemize}
In the pictures, the orientation of trivalent vertices is given by $\Orientation$. 
When it does not seem to cause confusion, we write $f_{vv'}$ for
$f_{vv'}^D$. We also call \emph{legs} the univalent vertices. For 
$k\in\N$, we call \emph{$k$--legs diagram} and \emph{$k_\tinf$--legs 
diagram} an $\BlMod$--colored diagram with, respectively, exactly and 
at most $k$ legs.
The \emph{degree} of a colored diagram is the number of trivalent
vertices of its underlying graph; the unique degree $0$ diagram is the
empty diagram. 

\begin{figure}
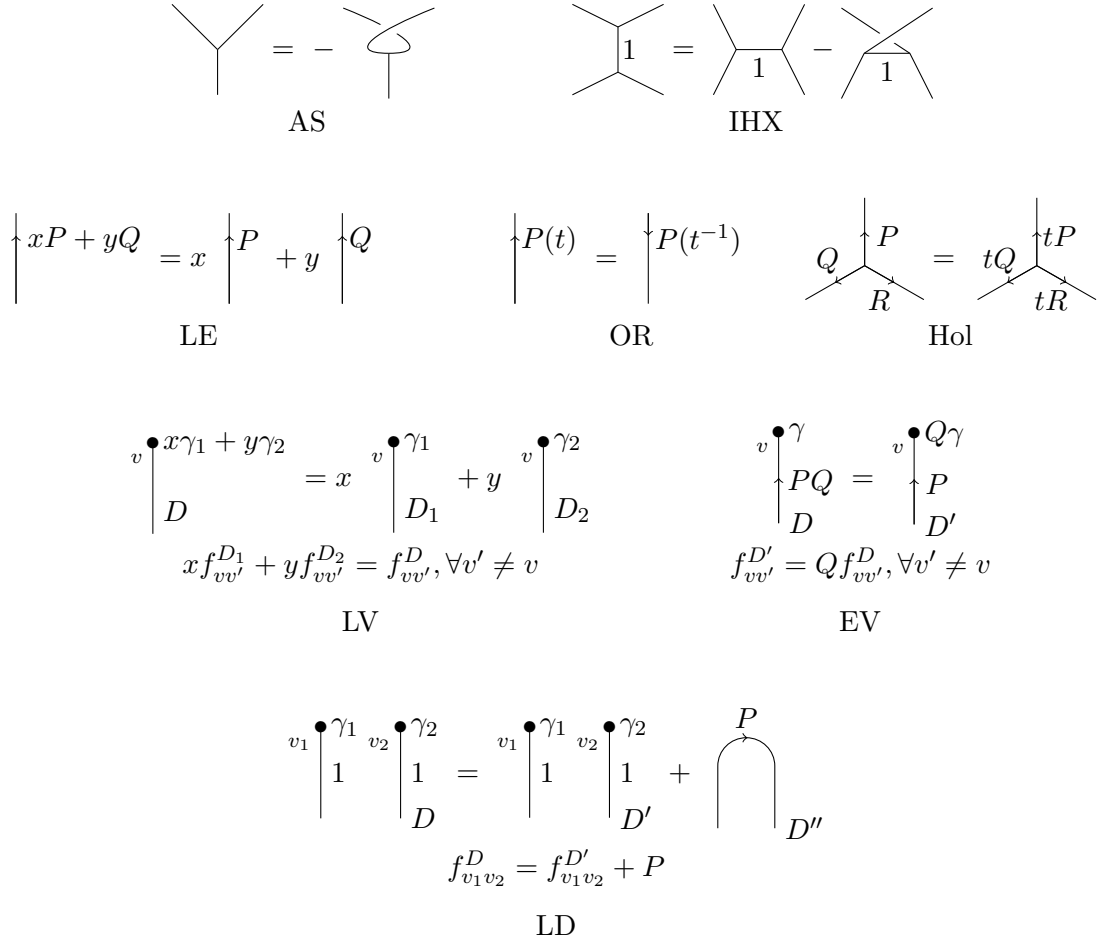

  \begin{gather*}
    \begin{array}{ccc}
      \ASU=\ -\ASD & \hspace{1cm} & \IHXU=\ \IHXD-\IHXT\\[.5cm]
      \AS&&\IHX
    \end{array}\\[.6cm]
    \begin{array}{ccccc}
      \LEU=x\ \ \LED+y\ \ \LET & \hspace{1cm} &
      \ORU=\ \ \ORD &&  \HolU=\ \HolD\\[.5cm]
      \LE && \OR &&\Hol
    \end{array}\\[.6cm]
    \begin{array}{ccc}
      \LVU=x\ \LVD+y\ \LVT & \hspace{1cm} &
      \EVU=\EVD\\[.5cm]
      xf_{vv'}^{D_1}+yf_{vv'}^{D_2}=f_{vv'}^{D},\forall v'\neq v
      && f_{vv'}^{D'}=Qf_{vv'}^{D},\forall v'\neq v\\[.3cm]
      \LV && \EV
    \end{array}\\[.6cm]
     \begin{array}{c}
      \LDU=\LDD+\ \ \LDT\\[.5cm]
      f_{v_1v_2}^{D}=f_{v_1v_2}^{D'}+P\\[.3cm]
      \LD
    \end{array}
  \end{gather*}
\caption{Relations on colored diagrams\\{\footnotesize In these 
    pictures, $x,y\in\Q$, $P,Q,R\in\Qt$ and $\gamma,\gamma_1,\gamma_2\in\Al$.}} \label{figrel}
\end{figure}

The automorphism group $\Aut\BlMod$ of the Blanchfield module $\BlMod$ acts on $\BlMod$--colored diagrams by evaluation of an automorphism on the labels of all the legs of a diagram simultaneously.
For $n\geq0$, we set: 
\[
\A_n\BlMod=\frac{\Q \big\langle\textnormal{$\BlMod$--colored
    diagrams of degree $n$}\big\rangle}{\Q \big\langle\AS, \IHX, \LE, \OR, \Hol, \LV, \EV, \LD,
    \Aut\big\rangle},
\] 
where the relations $\AS$ (anti-symmetry), $\IHX$, $\LE$ (linearity for edges), $\OR$ (orientation reversal), $\Hol$ (holonomy), 
$\LV$ (linearity for vertices), $\EV$ (edge-vertex) and $\LD$ (linking difference: this relation deals with the rational fractions associated to pairs of vertices) are described in Figure \ref{figrel} and $\Aut$ is
the set of relations $D=\zeta.D$ where $D$ is a $\BlMod$--colored diagram and $\zeta\in\Aut\BlMod$. 
Since the opposite of the identity is an automorphism of $\BlMod$, we have $\A_{2n+1}\BlMod=0$ for all $n\geq 0$.

\paragraph{Spaces of $\delta$--colored diagrams.}

Let $\delta\in\Qt$. A \emph{$\delta$--colored diagram} is a trivalent graph whose vertices 
are oriented and whose edges are oriented and labelled by $\frac{1}{\delta}\Qt$. The \emph{degree} of a $\delta$--colored diagram is the number of its vertices. 
For every integer $n\geq0$, set:
\[
\A_n(\delta)=\frac{\Q\big\langle\textnormal{$\delta$--colored diagrams of
    degree
    $n$}\big\rangle}{\Q\big\langle\AS,\IHX,\LE,\OR,\Hol,\Hol'\big\rangle},
\]
where the relation $\Hol'$ is represented in Figure \ref{fighol'} and the relations $\AS$, $\IHX$, $\LE$, $\OR$, $\Hol$ are represented in Figure \ref{figrel} but with edges now labelled in $\frac{1}{\delta}\Qt$. 
\begin{figure}
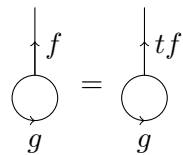

\[
\HolPU=\ \HolPD
\]
\caption{Relation $\Hol'$\\{\footnotesize In this picture, $f,g\in\frac{1}{\delta}\Qt$.}} \label{fighol'}
\end{figure}
Note that in the case of $\A_n\BlMod$, the relation $\Hol'$ is induced by the relations $\Hol$, $\EV$, $\LD$ and $\LV$, as shown in Figure~\ref{figrecovHolp}, where $\LV$ is used to see that one diagram is trivial at each application of $\LD$. 
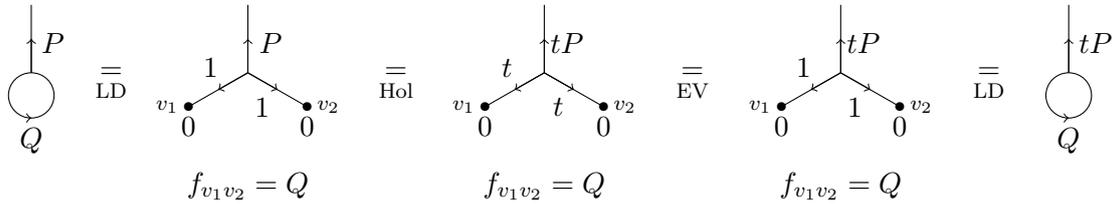
\begin{figure}
\begin{center}
\begin{tikzpicture}[scale=0.3]
 \begin{scope}
  \draw (0,0) -- (0,3);
  \draw (0,-1) circle (1);
  \draw[->] (-0.1,-2) -- (0,-2) node[below] {$Q$};
  \draw[->] (0,0) -- (0,1.5);
  \draw (0,1.3) node[right] {$P$};
 \end{scope}
 \draw (3.5,0) node {$=$} node[below] {\scriptsize LD};
 \begin{scope} [xshift=9.5cm]
  \edge{0}{$P$}
  \edge{120}{$1$}
  \edge{240}{$1$}
  \draw[rotate=120] (0,3) node {$\scriptstyle{\bullet}$} node[below] {$0$} node[left] {$\scriptstyle{v_1}$};
  \draw[rotate=240] (0,3) node {$\scriptstyle{\bullet}$} node[below] {$0$} node[right] {$\scriptstyle{v_2}$};
  \draw (0,-5) node {$f_{v_1v_2}=Q$};
 \end{scope}
 \draw (16,0) node {$=$} node[below] {\scriptsize Hol};
 \begin{scope} [xshift=22.5cm]
  \edge{0}{$tP$}
  \edge{120}{$t$}
  \edge{240}{$t$}
  \draw[rotate=120] (0,3) node {$\scriptstyle{\bullet}$} node[below] {$0$} node[left] {$\scriptstyle{v_1}$};
  \draw[rotate=240] (0,3) node {$\scriptstyle{\bullet}$} node[below] {$0$} node[right] {$\scriptstyle{v_2}$};
  \draw (0,-5) node {$f_{v_1v_2}=Q$};  
 \end{scope}
 \draw (29,0) node {$=$} node[below] {\scriptsize EV};
 \begin{scope} [xshift=35.5cm]
  \edge{0}{$tP$}
  \edge{120}{$1$}
  \edge{240}{$1$}
  \draw[rotate=120] (0,3) node {$\scriptstyle{\bullet}$} node[below] {$0$} node[left] {$\scriptstyle{v_1}$};
  \draw[rotate=240] (0,3) node {$\scriptstyle{\bullet}$} node[below] {$0$} node[right] {$\scriptstyle{v_2}$};
  \draw (0,-5) node {$f_{v_1v_2}=Q$};  
 \end{scope}
 \draw (42,0) node {$=$} node[below] {\scriptsize LD};
 \begin{scope} [xshift=45.5cm]
  \draw (0,0) -- (0,3);
  \draw (0,-1) circle (1);
  \draw[->] (-0.1,-2) -- (0,-2) node[below] {$Q$};
  \draw[->] (0,0) -- (0,1.5);
  \draw (0,1.3) node[right] {$tP$};
 \end{scope}
\end{tikzpicture}
\end{center} \caption{Recovering the relation Hol' from Hol, EV, LD and LV} \label{figrecovHolp}
\end{figure}

To an $\BlMod$--colored diagram $D$ of degree $n$, we associate a $\delta$--colored diagram $\psi_n(D)$ as follows. 
Denote by $V$ the set of legs of $D$. Define a \emph{pairing of $V$} as an involution of $V$ 
with no fixed point. For every such pairing $p$, define $D_p$ as the diagram obtained by replacing, in $D$, every pair $\big(v,p(v)\big)$ 
of associated legs---and their adjacent edges---by a colored edge as indicated in Figure \ref{figgroup}. 
\begin{figure}
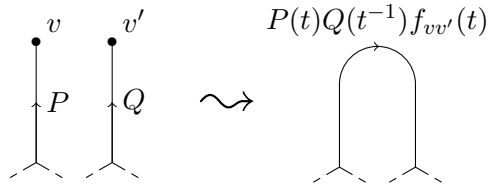
 
\[
\PairingU \hspace{.6cm}\raisebox{-1.5ex}{\huge $\leadsto$} \PairingD
\]
\caption{Pairing of two vertices} \label{figgroup}
\end{figure}
Now set:
\[
\psi_n(D)=\sum_{p\in\mathfrak{p}} D_p,
\]
where $\mathfrak{p}$ is the set of pairings of $V$. Note that, if $D$ has an odd number of legs, then $\mathfrak p$ is empty and $\psi_n(D)=0$. 
One can easily check that this assignment yields a well-defined $\Q$-linear map $\psi_n:\A_n\BlMod\to\A_n(\delta)$.

\subsection{Strategy} \label{secstrategy}

\paragraph{Getting rid of $\A_n(\delta)$.}
The map $\psi_n$ involves two diagram spaces defined by different kind of diagrams, namely $\BlMod$--colored diagrams and $\delta$--colored diagrams. 
The following result will allow us to work with $\BlMod$--colored diagrams only. 
\begin{theorem}[{\cite[Theorem 2.12]{M7}}] \label{th3n}
 Let $n$ and $N$ be non negative integers such that $N\geq\frac{3n}2$. Fix a Blanchfield module $\BlMod$ with annihilator $\delta$ and define the Blanchfield 
 module $\BlModN$ as the direct sum of $N$ copies of $\BlMod$. Then $\delta$ is also the annihilator of $\BlModN$ and the map 
 $\psib_n:\A_n\left(\BlModN\right)\to\A_n(\delta)$ is an isomorphism.
\end{theorem}
This result provides a rewritting of the map $\psi_n$ in the general
case. There is indeed a natural map
$\iota_n:\A_n\BlMod\to\A_n\left(\BlModN\right)$ defined on each
diagram by interpreting the labels of its legs as
elements of the first copy of $\BlMod$ in $\BlModN$, which makes the following diagram commute:
\[
\vcenter{\hbox{\begin{tikzpicture} [xscale=1.4,yscale=0.8]
 \draw (2,2) node {$\A_n\left(\BlModN\right)$};
 \draw (2,-2) node {$\A_n(\delta)$};
 \draw (0,0) node {$\A_n\BlMod$};
 \draw[->] (0.4,0.5) -- (1.6,1.5); \draw (0.9,1.3) node {$\iota_n$};
 \draw[->] (0.4,-0.5) -- (1.6,-1.5); \draw (0.9,-1.3) node {$\psi_n$};
 \draw[->] (2,1.5) -- (2,-1.5); \draw (2.3,0) node {$\psib_n$}; \draw (1.8,0) node {$\cong$};
\end{tikzpicture}}}.
\]
In particular, the injectivity of $\psi_n$ is equivalent to the injectivity of
$\iota_n$, what does not involve $\A_n(\delta)$ anymore. More generally, there is a natural map
$\iota_n^\ell:\A_n\left(\BlModl\right)\to\A_n\left(\BlModN\right)$ defined similarly to $\iota_n$. When it does not seem to cause confusion, $\iota_n^\ell$ is simply denoted $\iota_n$. When $n=2$, for every $N\geq3$, we have:
\[
\vcenter{\hbox{\begin{tikzpicture}
 \draw (-3.5,2) node {$\A_2\BlMod$} (0,2) node {$\A_2\big(\BlModD\big)$} (4,2) node {$\A_2\left(\BlModN\right)$};
 \draw[->] (-2.5,2) -- (-1.3,2); \draw[->] (1.3,2) -- (2.5,2); \draw (1.8,2) node[above] {$\iota_2^2$};
 \draw[->] (-3.2,2.4) .. controls +(2,1) and +(-2,1) .. (3.2,2.4); \draw (0,3.5) node {$\iota_2^1$};
 \draw[->] (-3.2,1.6) -- (-0.5,0.3); \draw[->] (0,1.6) -- (0,0.4); \draw[->] (3.2,1.6) -- (0.5,0.3);
 \draw (2.4,0.8) node {$\cong$}; \draw (-2.4,0.8) node {$\psi_2^1$}; \draw (0,1) node[right] {$\psi_2^2$};
 \draw (0,0) node {$\A_2(\delta)$};
\end{tikzpicture}}}.
\]
We focus on determining whether the maps $\iota_2^1$ and $\iota_2^2$ are injective or not. For that, it is sufficient to consider the case $N=3$.

\paragraph{Filtration by the number of legs.}
The second point in our strategy is to consider the filtration induced
by the number of legs. For $k=0,\dots,3n$, let $\A_n^{(k)}\BlMod$ be the subspace of $\A_n\BlMod$ generated by $k_\tinf$--legs diagrams and set:
\[
\widehat{\A}_n^{(k)}\BlMod=\frac{\Q \big\langle\textnormal{$k_\tinf$--legs diagrams
    of degree $n$}\big\rangle}{\Q \langle\AS,\IHX,\LE,\OR,\Hol,\LV,\EV,\LD,\Aut\rangle}.
\]
Recall that all these diagram spaces are trivial when $n$ is odd. Moreover, in a uni-trivalent graph, the numbers of univalent and trivalent vertices 
have the same parity, thus $\A_{2n}^{(2k+1)}\BlMod=\A_{2n}^{(2k)}\BlMod$ and $\widehat{\A}_{2n}^{(2k+1)}\BlMod\cong\widehat{\A}_{2n}^{(2k)}\BlMod$. 
Obviously, $\widehat{\A}_n^{(3n)}\BlMod=\A_n\BlMod= \A_n^{(3n)}\BlMod$. However, a subtlety of the
structure of the spaces $\A_n\BlMod$ is that the natural surjection $\widehat{\A}_n^{(k)}\BlMod\twoheadrightarrow\A_n^{(k)}\BlMod$ is not, in general, an isomorphism. 
A counterexample is given in Proposition \ref{propbasis} (\ref{item:Prop5}.\ref{item:Prop52}), which underlies the case where $\iota_2^2$ is not injective. 

\paragraph{Reduction of the presentations.}
To study the injectivity status of the map $\iota_2$, we first study the structure of the space $\A_2\left(\BlModT\right)$ to determine if 
$\A_2^{(k)}\left(\BlModT\right)$ is isomorphic to $\widehat{\A}_2^{(k)}\left(\BlModT\right)$ for $k=2,4$. If we have such isomorphisms, 
then Corollary~\ref{corinjectivity} states that the map $\iota_n$ is injective. Otherwise, we have to perform a similar study of the structure of $\A_2\BlMod$. 

To understand the structures of these diagram spaces, the strategy is to simplify the
given presentations by restricting simultaneously the set of generators and the set of
relations. This reduction process is initialized in Section \ref{secreduction} for a general 
Blanchfield module and pursued in the next sections for each specific case.

\section{Preliminary results} \label{secpreliminaries}

\subsection{Distributed diagrams}

We define notations that we will use throughout the rest of the paper. 
Let $\BlMod$ be a Blanchfield module with annihilator $\delta$. For a positive integer $N$, 
set $\displaystyle\BlModN=\bigoplus_{i=1}^N(\Al_i,\bl_i)$, where each $(\Al_i,\bl_i)$ is an isomorphic copy of $\BlMod$, 
given with a fixed isomorphism $\xi_i:\Al\to\Al_i$ that respects the Blanchfield pairing. Define the \emph{permutation automorphisms} $\xi_{ij}$ 
of $\BlModN$ as $\xi_j\circ\xi_i^{-1}$ on $\Al_i$, $\xi_i\circ\xi_j^{-1}$ on $\Al_j$ and identity on the other
$\Al_\ell$'s. Define $\Aut_\xi$ as the restriction of the $\Aut$ relation to these permutation automorphisms. 
Also denote by $\Aut_t$ and $\Aut_{-1}$ the restrictions of the $\Aut$ relation to the
automorphisms that are the multiplication by $t$ and $-1$ respectively on one $\Al_i$ and identity on the other $\Al_j$'s. 
If $\BlMod$ is cyclic, then define $\Aut_{res}$ as the union of
$\Aut_\xi$, $\Aut_t$ and $\Aut_{-1}$.
Otherwise, define $\Aut_{res}$ as the $\Aut$ relation restricted to permutation automorphisms and to automorphisms fixing one
$\Al_i$ setwise and the others pointwise.

Finally, for $\ell\geq0$, we say that an $\left(\BlModl\right)$--colored diagram $D$ is \emph{distributed} if there is a partition of the legs 
of $D$ into a disjoint union of pairs $\sqcup_{i\in I}\{v_i,w_i\}$ and an injective map $\sigma:I\to\{1,\dots,\ell\}$ 
such that the legs $v_i$ and $w_i$ are labelled in $\Al_{\sigma(i)}$ and the linking between vertices in different pairs is trivial.

\begin{proposition}[{\cite[Propositions 7.11 $\&$ 7.12]{M7}}] \label{proppresentation}
 For all non negative integers $n$, $k$ and $\ell$  such that $\ell\geq\frac k2$:
\[
\widehat{\A}_n^{(k)}\left(\BlModL\right)\cong\frac{\Q\Big\langle\textnormal{distributed
    $k_\tinf$--legs diagrams of degree $n$}\Big\rangle}
 {\Q
   \big\langle\AS,\IHX,\LE,\OR,\Hol,\LV,\EV,\LD,\Aut_{res}\big\rangle}.
\]
In particular, for all integers $N\geq\frac{3n}{2}$:
 \[
 \A_n\left(\BlModN\right)\cong\frac{\Q\Big\langle\textnormal{distributed
      $\left(\BlModN\right)$--colored diagrams of degree $n$}\Big\rangle}
 {\Q \big\langle\AS,\IHX,\LE,\OR,\Hol,\LV,\EV,\LD,\Aut_{res}\big\rangle}.
 \]
\end{proposition}

For positive integers $\ell_1\leq \ell_2$, 
let $\widehat{\iota}_n:\widehat{\A}_n^{(k)}\left(\BlMod^{\oplus\ell_1}\right)\to\widehat{\A}_n^{(k)}\left(\BlMod^{\oplus\ell_2}\right)$ be 
the natural map defined on each diagram by interpreting the labels of its legs as elements of the first $\ell_1$ copies of $\BlMod$ in
$\BlMod^{\oplus\ell_2}$. 

\begin{corollary} \label{corpresentation}
 For all non negative integers $n$, $k$, $\ell_1$ and $\ell_2$ such that $\ell_1,\ell_2\geq\frac k2$, the map
 $\widehat{\iota}_n:\widehat{\A}_n^{(k)}\left(\BlMod^{\oplus\ell_1}\right)\to\widehat{\A}_n^{(k)}\left(\BlMod^{\oplus\ell_2}\right)$ 
 is an isomorphism.
\end{corollary}
\begin{proof}
 A distributed $k_\tinf$--legs diagram involves at most $2k$ copies of $\Al$; up to $\Aut_\xi$, we can assume
 that these are copies whithin the first $\ell_1$ ones. Conclude with Proposition \ref{proppresentation}.
\end{proof}
  
The next lemma will be useful in particular to restrict the study of the map $\iota_2$ to suitable quotients. 
\begin{corollary} \label{coridealpres}
 Let $n$, $N$, $k$ and $\ell$ be non negative integers such that $N\geq\frac{3n}{2}$ and $\frac k2\leq \ell\leq N$. 
 If $\A_n^{(k)}\left(\BlModN\right)\cong\widehat{\A}_n^{(k)}\left(\BlModN\right)$,
 then the map $\A_n^{(k)}\big(\BlModL\big)\to\A_n^{(k)}\left(\BlModN\right)$ induced by $\iota_n$ is an isomorphism. 
\end{corollary}
\begin{proof}
By Corollary \ref{corpresentation}, the map $\widehat{\iota}_n:\widehat{\A}_n^{(k)}\big(\BlModL\big)\to\widehat{\A}_n^{(k)}\left(\BlModN\right)$ 
is an isomorphism. Hence we have the following commutative diagram:
\[
 \vcenter{\hbox{\begin{tikzpicture}
  \draw (0,2) node {$\widehat{\A}_n^{(k)}\big(\BlModl\big)$} (0,0) node {$\A_n^{(k)}\big(\BlModl\big)$}
        (7,1) node {$\A_n^{(k)}\big(\BlModN\big)$};
  \draw[->] (2,1.8) -- (5,1.2); \draw (3.5,1.5) node[above] {$\cong$};
  \draw[->] (2,0.2) -- (5,0.8); \draw[->>] (0,1.5) -- (0,0.5);
 \end{tikzpicture}}}.
\]
The statement follows.
\end{proof}

\begin{lemma} \label{lemmaiotaquotients}
 Let $n$, $k$, $\ell_1$ and $\ell_2$ be non negative integers such
 that $\ell_1\leq \ell_2$ and $\frac k2\leq \ell_2$. 
 Let $\widetilde{\A}_n^{(k)}$ denote the image of $\widehat{\A}_n^{(k)}$ in $\widehat{\A}_n^{(k+2)}$. 
 Then the map $\fract{\widehat{\A}_n^{(k+2)}(\BlMod^{\oplus \ell_1})}{\widetilde{\A}_n^{(k)}(\BlMod^{\oplus \ell_1})}\to
 \fract{\widehat{\A}_n^{(k+2)}(\BlMod^{\oplus\ell_2})}{\widetilde{\A}_n^{(k)}(\BlMod^{\oplus \ell_2})}$ induced by $\widehat{\iota}_n$ is injective.
\end{lemma}
\begin{proof}
 Let us define a left inverse of $\widehat{\iota}_n$. Let $D$ be a
 distributed $(k+2)_\tinf$--legs diagram. For each leg colored by $\eta\in\Al_i$ 
 with $\ell_1<i\leq\ell_2$, replace the label by $\xi_1\circ\xi_i^{-1}(\eta)$. Choose any linkings coherent with these new labels. Thanks to the relation $\LD$, 
 any such choice defines the same class $\sigma_n(D)$ in the quotient $\fract{\widehat{\A}_n^{(k+2)}(\BlMod^{\oplus \ell_1})}{\widetilde{\A}_n^{(k)}(\BlMod^{\oplus \ell_1})}$. 
 This provides a well-defined map $\sigma_n:\fract{\widehat{\A}_n^{(k+2)}(\BlMod^{\oplus \ell_2})}{\widetilde{\A}_n^{(k)}(\BlMod^{\oplus \ell_2})}\to
 \fract{\widehat{\A}_n^{(k+2)}(\BlMod^{\oplus\ell_1})}{\widetilde{\A}_n^{(k)}(\BlMod^{\oplus \ell_1})}$ such that $\sigma_n\circ\widehat{\iota}_n=Id$. 
\end{proof}

\begin{corollary} \label{corinjectivity}
 Let $n$, $\ell$ and $N$ be non negative integers such that $n$ is even, $\ell\leq N$ and $N\geq\frac{3n}{2}$. 
 If
 $\widehat{\A}_n^{(2k)}\left(\BlModN\right)\cong\A_n^{(2k)}\left(\BlModN\right)$
 for all integers $k$ such that $\ell\leq k\leq \frac{3n}2$, 
 then the map $\iota_n:\A_n\big(\BlMod^{\oplus \ell}\big)\to\A_n\big(\BlModN\big)$ is injective. 
 Moreover, it implies that $\widehat{\A}_n^{(2k)}\big(\BlModl\big)\cong\A_n^{(2k)}\big(\BlModl\big)$ for all $k\geq0$. 
\end{corollary}
\begin{proof}
 We prove by induction on $k$ that $\widehat{\A}_n^{(2k)}\big(\BlModl\big)\cong\widetilde{\A}_n^{(2k)}\big(\BlModl\big)\cong\A_n^{(2k)}\big(\BlModl\big)$ 
 and that the map $\A_n^{(2k)}\big(\BlModl\big)\to\A_n^{(2k)}\big(\BlModN\big)$ induced by $\iota_n$ is injective. 
 For $k\leq\ell$, Corollary~\ref{corpresentation} says that $\widehat{\iota}_n:\widehat{\A}_n^{(2k)}\big(\BlModl\big)\to\widehat{\A}_n^{(2k)}\big(\BlModN\big)$ 
 is an isomorphism. For $k>\ell$, we use the following observation.
 \begin{fact*}
  Let $f:E_1\to E_2$ be a morphism between two vector spaces. Let $F_1\subset E_1$ and $F_2\subset E_2$ be linear subspaces such that $f(F_1)\subset F_2$ and let $\bar{f}:\fract{E_1}{F_1}\to \fract{E_2}{F_2}$ be the map induced by $f$. If $\bar{f}$ and $f_{|F_1}$ are injective, then $f$ is injective.
 \end{fact*}
 Together with Lemma~\ref{lemmaiotaquotients} and the induction hypothesis, this implies that the map $\widehat{\iota}_n:\widehat{\A}_n^{(2k)}\big(\BlModl\big)\to\widehat{\A}_n^{(2k)}\big(\BlModN\big)$ is injective. 
 In both cases, we get the following commutative diagram:
 \[
 \vcenter{\hbox{\begin{tikzpicture}
  \draw (0,2) node {$\widehat{\A}_n^{(2k)}\big(\BlModl\big)$} (0,0) node {$\widetilde{\A}_n^{(2k)}\big(\BlModl\big)$}
        (0,-2) node {$\A_n^{(2k)}\big(\BlModl\big)$} (7,0) node {$\A_n^{(2k)}\big(\BlModN\big)$,};
  \draw[->] (2,1.6) -- (5,0.3); \draw (2,1.6) arc (240:60:0.1);
  \draw[->] (2,0) -- (5,0); \draw[->] (2,-1.6) -- (5,-0.3);
  \draw[->>] (0,1.5) -- (0,0.5); \draw[->>] (0,-0.5) -- (0,-1.5);
 \end{tikzpicture}}}
 \]
which concludes the proof.
\end{proof}

\subsection{First reduction of the presentations} \label{secreduction}

\paragraph{Getting rid of lollipops.}

We start with a lemma on 0--labelled vertices. 
\begin{lemma} \label{lem:0LabelledVertex}
 If $D$ is an $\BlMod$--colored diagram with a 0--labelled vertex $v$, then
 \[
 D=\sum_{\substack{v'\textnormal{ vertex of }D\\v'\neq v}}D_{vv'},
 \]
 where $D_{vv'}$ is obtained from $D$ by pairing $v$ and $v'$ as in Figure \ref{figgroup}.
\end{lemma}
\begin{proof}
 Since the vertex $v$ is labelled by $0$, the linking $f_{vv'}$ is a polynomial for any vertex $v'\neq v$. The conclusion follows using the relations $\LD$ and $\LV$.
\end{proof}

Now, the following lemma reduces the set of generators.
\begin{lemma} \label{lemmaloop}
 The general presentation of $\A_n\BlMod$ and the presentations of $\widehat{\A}_n^{(k)}\left(\BlModl\right)$ and $\A_n\left(\BlModN\right)$ 
 given in Proposition \ref{proppresentation} are still valid when removing from the generators the diagrams whose underlying graph contains 
 a connected component $\GenLoop$.
\end{lemma}
\begin{proof}
 Thanks to the $\OR$ relation, such a diagram can be written
 \[
D=\GLoop{\eta}{Q(t)}{P(t)}\sqcup D'.
\]
 Writing $\delta=\sum_{k=p}^q a_k t^k$, we have:
 \[
D=\frac1{\delta(1)}\sum_{k=p}^qa_k\left(\GLoop{\eta}{t^kQ(t)}{P(t)}\sqcup D'\right)
 =\frac1{\delta(1)}\left(\GLoop{\delta(t)\eta}{Q(t)}{P(t)}\sqcup D'\right)
 =\frac1{\delta(1)}\left(\GLoop{0}{Q(t)}{P(t)}\sqcup D'\right),
\]
 where the first equality holds since each diagram in the sum is equal to $D$ by $\Hol'$ and the second equality follows from $\EV$ and $\LV$. 
 Then, using Lemma \ref{lem:0LabelledVertex}, $D$ can be written as a sum of diagrams with less legs. 
 Check that all the relations involving $D$ can be recovered from relations on diagrams with less legs. Conclude by decreasing induction 
 on the number of legs.
\end{proof}

Finally, we state a corollary of Lemma \ref{lem:0LabelledVertex} which will be useful later.
\begin{corollary} \label{cor:DV}
 Let $D$ be an $\BlMod$--colored diagram and let $v$ be a univalent vertex of $D$. If the annihilator of $\Al$ is $\delta=t+a+t^{-1}$, then
 \[
 D_+=-aD-D_-+\sum_{\substack{v'\textnormal{ vertex of }D\\v'\neq v}}D_{vv'},
 \]
 where $D_+$ and $D_-$ are obtained from $D$ by multiplying the label of $v$ and the linkings $f_{vv'}$ by $t$ and $t^{-1}$ respectively, 
 and $D_{vv'}$ is obtained from $D$ by pairing $v$ and $v'$ as in Figure \ref{figgroup}.
\end{corollary}

\paragraph{Taming 6 and 4--legs generators.}

We now give two lemmas that initialize the reduction process announced
in Section \ref{secstrategy}.
For that, define \emph{$\YY$--diagrams} similarly as $(\Al,\bl)$--colored diagrams with
underlying graph $\SixLegsPattern$, except that edges are neither
oriented nor labelled. Thanks to $\OR$, those can be thought of as honest
$(\Al,\bl)$--colored diagrams with edges labelled by 1 and oriented
arbitrarily.
Define also $\bHol$ as the relations given in Figure \ref{figbHol}; 
note that $\bHol$ is easily deduced from $\Hol$ and $\EV$.
\begin{figure}[htb] 
\begin{center}
 \begin{tikzpicture} [scale=0.15]
  \draw (-3.4,-2) -- (0,0) -- (0,4) (0,0) -- (3.4,-2);
  \draw (0,4) node {$\scriptstyle{\bullet}$} node[left] {$\eta_1$} node[right] {$\scriptstyle{v_1}$};
  \draw (-3.4,-2) node {$\scriptstyle{\bullet}$} node[below] {$\eta_2$} node[above] {$\scriptstyle{v_2}$};
  \draw (3.4,-2) node {$\scriptstyle{\bullet}$} node[below] {$\eta_3$} node[above] {$\scriptstyle{v_3}$};
  \draw (6,-8) node {$D$};
 \begin{scope} [xshift=11.5cm]
  \draw (-3.4,-2) -- (0,0) -- (0,4) (0,0) -- (3.4,-2);
  \draw (0,4) node {$\scriptstyle{\bullet}$} node[left] {$\eta_4$} node[right] {$\scriptstyle{w_1}$};
  \draw (-3.4,-2) node {$\scriptstyle{\bullet}$} node[below] {$\eta_5$} node[above] {$\scriptstyle{w_2}$};
  \draw (3.4,-2) node {$\scriptstyle{\bullet}$} node[below] {$\eta_6$} node[above] {$\scriptstyle{w_3}$};
 \end{scope}
 \draw (20,0) node {$=$};
 \begin{scope} [xshift=29cm]
  \draw (-3.4,-2) -- (0,0) -- (0,4) (0,0) -- (3.4,-2);
  \draw (0,4) node {$\scriptstyle{\bullet}$} node[left] {$t\eta_1$} node[right] {$\scriptstyle{v_1}$};
  \draw (-3.4,-2) node {$\scriptstyle{\bullet}$} node[below] {$t\eta_2$} node[above] {$\scriptstyle{v_2}$};
  \draw (3.4,-2) node {$\scriptstyle{\bullet}$} node[below] {$t\eta_3$} node[above] {$\scriptstyle{v_3}$};
 \end{scope}
 \draw (35,-8) node {$D'$};
 \begin{scope} [xshift=40.5cm]
  \draw (-3.4,-2) -- (0,0) -- (0,4) (0,0) -- (3.4,-2);
  \draw (0,4) node {$\scriptstyle{\bullet}$} node[left] {$\eta_4$} node[right] {$\scriptstyle{w_1}$};
  \draw (-3.4,-2) node {$\scriptstyle{\bullet}$} node[below] {$\eta_5$} node[above] {$\scriptstyle{w_2}$};
  \draw (3.4,-2) node {$\scriptstyle{\bullet}$} node[below] {$\eta_6$} node[above] {$\scriptstyle{w_3}$};
 \end{scope}
 \draw (60,0) node {$f_{v_iw_j}^{D'}=tf_{v_iw_j}^D$};
 \end{tikzpicture}
\end{center}
 \caption{The relation $\bHol$} \label{figbHol}
\end{figure}

\begin{lemma} \label{lemma2step6legs}
 The space $\A_2(\Al,\bl)$ admits the presentation with:
 \begin{itemize}
  \item as generators: $\YY$--diagrams and all $4_\tinf$--legs diagrams;
  \item as relations: $\AS$, $\LV$, $\LD$, $\Aut$ and $\bHol$ on all
    generators and $\IHX$, $\LE$, $\Hol$, $\OR$ and $\EV$ on $4_\tinf$--legs generators.
 \end{itemize}
 
 The space $\A_2\big(\BlModT\big)$ admits the similar presentation with
 generators restricted to distributed $\big(\BlModT\big)$--colored diagrams and the relation $\Aut$ restricted to $\Aut_{res}$.
\end{lemma}
\begin{proof}
 Any degree two $(\Al,\bl)$--colored diagram with six legs has underlying graph $\SixLegsPattern$. Using $\LE$, 
 any such diagram can be written as a $\Q$--linear combination of diagrams having all edges labelled by powers of $t$. 
 Then, using $\OR$ and $\EV$, these powers of $t$ can be pushed to the legs. This produces a canonical decomposition of any 6--legs
 diagram in terms of $\YY$--diagrams. Hence it provides a $\Q$--linear map from the $\Q$--vector space freely generated by all $(\Al,\bl)$--colored diagrams 
 of degree 2 to the module $\A'_2(\Al,\bl)$ defined by the presentation given in the statement. 
 This map descends to a well-defined map $\tau$ from $\A_2(\Al,\bl)$ to $\A'_2(\Al,\bl)$. 
 Indeed, it is sufficient to check that all generating relation in $\A_2(\Al,\bl)$ is sent to zero. It
 is immediate for $\AS$, $\LE$, $\OR$, $\LV$, $\LD$ and $\Aut$; it is true for $\EV$ and $\Hol$ by applying $\LV$ and
 $\bHol$ respectively on the image; it also holds for $\IHX$ since there is no such
 relation involving diagrams with underlying graph $\SixLegsPattern$.

 Now, it is clear that sending a diagram to itself gives a well-defined map $\A'_2(\Al,\bl)\to\A_2(\Al,\bl)$ which is the inverse of $\tau$.
\end{proof}

Now, we address the case of 4--legs generators. For that, we define 
\emph{$\H$--diagrams} similarly as $(\Al,\bl)$--colored diagrams with
underlying graph $\FourLegsPattern$, except that edges are neither
oriented nor labelled. Again, thanks to $\OR$, those can be thought of as honest
$(\Al,\bl)$--colored diagrams with edges labelled by 1 and oriented arbitrarily.

\begin{lemma} \label{lemma2step4legs}
 The space $\widehat{\A}_2^{(4)}(\Al,\bl)$ admits the presentation
 with:
 \begin{itemize}
  \item as generators: $\H$--diagrams and all $2_\tinf$--legs diagrams;
  \item as relations: $\AS$, $\IHX$, $\LV$, $\LD$ and $\Aut$ on all generators and $\LE$, $\Hol$, $\OR$ and $\EV$ on $2_\tinf$--legs generators.
 \end{itemize}
 The space $\widehat{\A}^{(4)}_2\big(\BlModT\big)$ admits the similar presentation with
 generators restricted to distributed $\big(\BlModT\big)$--colored diagrams and the relation $\Aut$ restricted to $\Aut_{res}$.
\end{lemma}
\begin{proof}
 First use Lemma \ref{lemmaloop} to reduce the 4--legs generators
 to those with underlying graph $\FourLegsPattern$ and then proceed as
 in the previous lemma. Here, the relation $\Hol$ is also needed
 to remove the power of $t$ from the central edge and the obtained
 decomposition is not anymore canonical. 
 However, two possible decompositions are related by the relation of $\Aut$
 associated with the automorphism that multiplies the whole Blanchfield module by $t$. 
\end{proof}

\paragraph{Taming leg labels.}

Now, we want to go further in the reduction of the presentations. Fix a $\Q$-basis $\omega$ of $\Al$. For all $\gamma,\eta\in\omega$, 
fix $f(\gamma,\eta)\in\Q(t)$ such that
$\bl(\gamma,\eta)=f(\gamma,\eta)\ mod\ \Qt$.
For $\ell\geq1$, identify $\BlModl$ with $\oplus_{1\leq i\leq \ell}(\Al_i,\bl_i)$ 
and let $\Omega$ be the union of the $\xi_i(\omega)$ for $i=1,\dots,\ell$. 
An $\BlModl$--colored diagram (resp. $\YY$--diagram, $\H$--diagram) is called \emph{$\omega$--admissible}, or
simply \emph{admissible} when there is no ambiguity on $\omega$, if:
\begin{enumerate}[label=(\roman*)]
 \item its legs are colored by elements of $\Omega$, 
 \item\label{item:linking} for two vertices $v$ and $w$ that are respectively colored by $\xi_i(\gamma)$ and $\xi_j(\eta)$, $f_{vw}=f(\gamma,\eta)$ 
  if $i=j$ and $f_{vw}=0$ otherwise.
\end{enumerate}
Every $\BlModl$--colored diagram (resp. $\YY$--diagram, $\H$--diagram) $D$ has a canonical \emph{$\omega$--reduction}, 
which is the decomposition as a $\Q$--linear sum of $\omega$--admissible diagrams obtained as follows. Write all the labels
of the legs as $\Q$--linear sums of elements of $\Omega$. Then use $\LV$ to write $D$ as a $\Q$--linear sum of diagrams 
with legs labelled by $\Omega\cup\{0\}$ and the $\Omega$-labelled legs satisfying Condition \ref{item:linking}. 
Finally, apply repeatedly Lemma \ref{lem:0LabelledVertex} to remove $0$--labelled vertices.

In the next step, we will not be able to reduce further the sets of generators and relations without rewriting some of the relations first. 
Denote by $\Aut^\omega$ the set of relations $D=\Sigma$ where $D$ is an $\omega$--admissible diagram and $\Sigma$ is the $\omega$--reduction 
of $\zeta.D$ for $\zeta\in\Aut\BlMod$. Define similarly $\Aut_{res}^\omega$ and $\Aut_t^\omega$. Define $\bHol^\omega$ as the set of relations 
that identify an $\omega$--admissible diagram $D$ with the $\omega$--reduction of the corresponding diagram $D'$ of Figure \ref{fighol'}. 

In general, if a family of generators is given for the group $\Aut(\Al,\bl)$, then the $\Aut$ relations, as well as the $\Aut^\omega$ relations, 
can be restricted to the set of relations provided by the automorphisms of this generating family.

\begin{lemma} \label{lemma3step6legs}
 The space $\A_2\BlMod$ admits the presentation with:
 \begin{itemize}
  \item as generators: $\omega$--admissible $\YY$--diagrams and all $4_\tinf$--legs diagrams;
  \item as relations: $\AS$, $\Aut^\omega$ and $\bHol^\omega$ on $6$--legs generators and $\AS$, $\IHX$, $\Hol$, $\LE$, $\OR$,
   $\LV$, $\LD$, $\EV$ and $\Aut$ on $4_\tinf$--legs generators. 
 \end{itemize}
 The space $\A_2\big(\BlModT\big)$ admits the similar
 presentation with generators restricted to distributed
 $\big(\BlModT\big)$--colored diagrams and the relations $\Aut^\omega$
 restricted to $\Aut^\omega_{res}$. If $\Al$ is cyclic, $\Aut_{res}^\omega$
 can be replaced by the union of $\Aut_\xi$ and $\Aut_t^\omega$.
\end{lemma}
\begin{proof}
 Starting from the presentation given in Lemma \ref{lemma2step6legs} and using the $\omega$--reduction, 
 one can proceed as in the proof of Lemma \ref{lemma2step6legs}.
 The only difficulty is to prove that the $\omega$--reduction of all $\Aut$ and $\bHol$ relations are indeed zero in
 the new presentation. To see that for $\Aut$, consider a relation $D=\zeta.D$ for an $\BlMod$--colored
 diagram $D$ and an automorphism $\zeta\in\Aut\BlMod$. Let $D=\sum_i\alpha_i D_i$ be the $\omega$--reduction of $D$. 
 For each $i$, write $\zeta.D_i=\sum_s\beta^i_sD^i_s$ the $\omega$--reduction of the diagram $\zeta.D_i$. 
 Check that $\zeta.D=\sum_i\alpha_i\sum_s\beta^i_sD^i_s$ is the $\omega$--reduction of $\zeta.D$. It follows that 
 the relation $D=\zeta.D$ is sent onto a $\Q$--linear combination of the relations $D_i=\sum_s\beta^i_sD^i_s$, 
 which are in $\Aut^\omega$. Relations $\bHol$ can be handled similarly.

 For the last assertion, note that the relation $\Aut_\xi$ never identifies an admissible diagram with a non-admissible one and that the relation $\Aut_{-1}$ 
 on admissible distributed diagrams only induces trivial relations.
\end{proof}

For the reduction of the $4$--legs generators, we focus on the
$\BlModT$ case and we introduce a more restrictive notion of admissible diagrams. 
An $\omega$--admissible $\H$--diagram is \emph{strongly $\omega$--admissible},
or simply \emph{strongly admissible} when there is no ambiguity on $\omega$, if its legs are colored in $\Al_1$ and $\Al_2$ and if two legs adjacent to 
a same trivalent vertex are labelled in different $\Al_i$'s.

\begin{lemma} \label{lemma3step4legs}
 The space $\widehat{\A}^{(4)}_2\big(\BlModT\big)$ admits the presentation with:
 \begin{itemize}
  \item as generators: strongly $\omega$--admissible $\H$--diagrams and all $2_\tinf$--legs diagrams;
  \item as relations: $\AS$ and $\Aut^\omega_{res}$ on $4$--legs generators and $\AS$, $\IHX$, $\LE$, $\Hol$, $\OR$, $\LV$, $\LD$, $\EV$ and $\Aut$
   on $2_\tinf$--legs generators.
 \end{itemize}
If $\Al$ is cyclic, $\Aut_{res}^\omega$ can be replaced by the union of $\Aut_\xi$ and $\Aut_t^\omega$.
\end{lemma}
\begin{proof}
Via at most one $\Aut_\xi$ relation, any $\omega$--admissible $\H$--diagram 
is equal to an $\omega$--admissible $\H$--diagram whose legs are labelled by $\Al_1$ and $\Al_2$. Moreover, if $\gamma_1,\eta_1\in\Al_1$ and 
$\gamma_2,\eta_2\in\Al_2$, then the $\IHX$ relation gives:
 \[
 \fourlegsop{\gamma_1}{\eta_1}{\gamma_2}{\eta_2}{0.12}=\fourlegsop{\gamma_1}{\gamma_2}{\eta_1}{\eta_2}{0.12}-\fourlegsop{\gamma_1}{\eta_2}{\eta_1}{\gamma_2}{0.12}.
 \]
It follows that any $\H$--diagram has a canonical decomposition
in terms of strongly $\omega$--admissible $\H$--diagrams. Proceed then
as in the proof of Lemma \ref{lemma3step6legs}.
\end{proof}

A set $\E$ of $\omega$--admissible $\YY$--diagrams (resp. $\H$--diagrams) is \emph{essential} if any $\omega$--admissible
$\YY$--diagram (resp. $\H$--diagram) which is not in $\E$ is either equal to a diagram in $\E$ via an $\AS$ or $\Aut_\xi$ relation,
or trivial by $\AS$. Denote by $\Aut^\E$ the set of relations $D=\Sigma$, where $D$ is an element of $\E$ and $\Sigma$ is the $\omega$--reduction 
of $\zeta.D$ for some $\zeta\in\Aut\BlMod$, rewritten in terms of $\E$. Define similarly $\bHol^\E$ and
$\Aut_*^\E$, where $\Aut_*$ is any subfamily of $\Aut$ described as the relations arising from the action of a subset of
$\Aut\BlMod$---for instance $\Aut_{res}$ or $\Aut_t$.

\begin{lemma} \label{lemmacheckrelations}
 If $\E$ is an essentiel set of $\omega$--admissible $\YY$--diagrams (resp. $\H$--diagrams), then the $\YY$--diagrams
 (resp. $\H$--diagrams) in the set of generators of the presentation given in Lemma~\ref{lemma3step6legs} (resp. Lemma~\ref{lemma3step4legs}) 
 can be restricted to $\E$ and the relations $\Aut^\omega$, $\Aut_{res}^\omega$, $\Aut_t^\omega$ and $\bHol^\omega$ can be
 replaced by $\Aut^\E$, $\Aut_{res}^\E$, $\Aut_t^\E$ and $\bHol^\E$ respectively. Moreover, if $\E$ is minimal, then $\AS$
 and $\Aut_\xi$ on $\YY$--diagrams (resp. $\H$-diagrams) can be removed from the set of relations.
\end{lemma}
\begin{proof}
 If an $\omega$--admissible diagram is trivial by $\AS$, then a relation $\bHol$
 or $\Aut$ involving this diagram gives a trivial relation; indeed, the terms 
 in the corresponding decomposition are trivial or cancel by pairs. Similarly, if two $\omega$--admissible diagrams are related by a relation $\AS$, then 
 the relations $\bHol$ and $\Aut$ applied to these diagrams provide
 the same relations.

 If $D$ is an $\omega$--admissible diagram and $D'=\xi_{ij}.D$ for some permutation automorphism $\xi_{ij}$, then any $\bHol$
 relation involving $D'$ is recovered from the action of $\xi_{ij}$ on the corresponding $\bHol$ relation involving $D$, and the relation
 resulting from the action of some automorphism $\zeta$ on $D'$ is recovered by the action of  $\xi_{ij}\circ\zeta\circ\xi_{ij}$ on $D$.

 For the last assertion, it is sufficient to notice that an $\AS$ relation makes either two generators to be equal, or a generator to be
 trivial, and that an $\Aut_\xi$ relation always identifies two generators.
\end{proof}

At this point, we have reduced the presentation for $\A_2\BlMod$ so
that we only have to consider non ($\AS$ and $\Aut_\xi$--trivially)
redundant $\YY$--diagrams with prescribed rational fractions on pairs of vertices depending
only on the labels, which are all in a given $\Q$--basis
of $\Al$, and $4_\tinf$--legs diagrams; the $\YY$--diagrams being only
subject to $\Aut$ and $\bHol$ relations rewritten in these
$\YY$--diagrams.

A similar reduction has been done for
$\A_2\big(\BlModT\big)$, where $\Aut$ is even replaced by
$\Aut_{res}$; if $\Al$ is cyclic, the latter can further be replaced
by $\Aut_t$.
Likewise, the presentation for $\widehat{\A}^{(4)}_2\big(\BlModT\big)$
has been reduced so
that we only have to consider non ($\AS$ and $\Aut_\xi$--trivially)
redundant $\H$--diagrams with prescribed rational fractions on pairs of vertices depending
only on the labels, which are all in a given $\Q$--basis
of $\Al$, and $2_\tinf$--legs diagrams; the $\H$--diagrams being only
subject to $\Aut_{res}$ relations rewritten in these
$\H$--diagrams; if $\Al$ is cyclic, $\Aut_{res}$ can further be replaced
by $\Aut_t$.

\section{Case when $\Al$ is of $\Q$--dimension two and cyclic}

In this section, we assume that $\Al$ is a cyclic Blanchfield module of
$\Q$--dimension two. Let $\delta=t+a+t^{-1}$ be its annihilator; note that $a\neq -2$.
Let $\gamma$ be a generator of $\Al$. Since the pairing $\bl$ is hermitian and non degenerate, we can set 
$\bl(\gamma,\gamma)=\frac{r}{\delta}\ mod\ \Qt$ with $r\in\Q^*$. Throughout this
section, we fix the basis $\omega$ to be $\{\gamma,t\gamma\}$ and we set $f(t^{\e_1}\gamma,t^{\e_2}\gamma)=t^{\e_1-\e_2}\frac r\delta$, where
$\e_1,\e_2\in\{0,1\}$. Accordingly, set $\gamma_i=\xi_i(\gamma)$ for $i=1,2,3$.

\subsection{Structure of $\A_2\left(\BlModT\right)$}
\label{secbasis}

The main results of this section are gathered in the following proposition.
\begin{proposition} \label{propbasis}
If $\BlMod$ is a cyclic Blanchfield module of $\Q$--dimension two with
annihilator $t+a+t^{-1}$, then:
\begin{enumerate}
\item\label{item:Prop1} $\A_2^{(2)}\left(\BlModT\right)\cong\widehat{\A}_2^{(2)}\left(\BlModT\right)$;
\item\label{item:Prop2} $\fract{\A_2\left(\BlModT\right)}{\A_2^{(2)}\left(\BlModT\right)}$ is freely generated by the diagrams $H_1$ and $G_1$ of Figure \ref{figbasis};
\item\label{item:Prop3} the natural map $\widehat{\A}_2^{(2)}\left(\BlModT\right)\to\widehat{\A}_2^{(4)}\left(\BlModT\right)$ is injective 
 and the corresponding quotient 
 $\fract{\widehat{\A}_2^{(4)}\left(\BlModT\right)}{\widehat{\A}_2^{(2)}\left(\BlModT\right)}$
 is freely generated by the $\H$--diagrams $H_1$ and $H_3$ given in 
 Figure \ref{figbasis}; 
\item\label{item:Prop4} if $a\neq1$, then $\A_2\left(\BlModT\right)=\A_2^{(4)}\left(\BlModT\right)\cong\widehat{\A}_2^{(4)}\left(\BlModT\right)$;
\item\label{item:Prop5} if $a=1$, then
  \begin{enumerate}[label=\roman*.]
  \item\label{item:Prop51}
    $\A_2^{(4)}\left(\BlModT\right)\varsubsetneq\A_2\left(\BlModT\right)$
    and the quotient
    $\fract{\A_2^{(4)}\left(\BlModT\right)}{\A_2^{(2)}\left(\BlModT\right)}$
    is freely generated by the $\H$--diagram $H_1$ given in Figure \ref{fig:Haches};
  \item\label{item:Prop52}
    $\A_2^{(4)}\left(\BlModT\right)\ncong\widehat{\A}_2^{(4)}\left(\BlModT\right)$.
  \end{enumerate}
\end{enumerate}
\end{proposition}

\begin{figure}
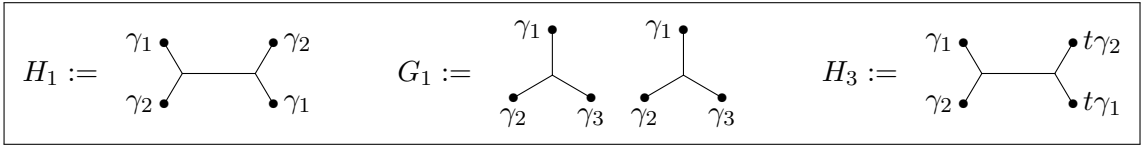

\[\fbox{
$H_1:=\fourlegs{}{}{}{}
\hspace{1cm}
G_1:=\sixlegs{}{}{}{}{}{}
\hspace{1cm}
H_3:=\fourlegs{}{}tt$}
\]
\caption{
Some generators for our diagram spaces\\
        {\footnotesize In these pictures, all edges are labelled by $1$ and the linkings are given}\\
        {\footnotesize by $f_{vw}=r/\delta$ when $v$ and $w$ are labelled by the same $\gamma_i$ and $0$ otherwise.}} 
\label{figbasis}
\end{figure}

The proof of this proposition will derive from the next results, which
resume the reduction process where it was left at the end of Section
\ref{secreduction}. In order to make the text easier, we will denote by
$\Aut^i_t$,
for any $i\in\{1,2,3\}$, the $\Aut_t$ relation applied on $\Al_i$.

\begin{lemma} \label{lemma4step6legs}
 The space $\A_2\big(\BlModT\big)$ admits the presentation with:
 \begin{itemize}
  \item as generators: the $\YY$--diagrams $D_1$, $D_2$ of Figure \ref{fig6legs1} and $G_1$, $G_2$, $G_3$, $G_4$ of Figure \ref{fig6legs2} and 
   all $4_\tinf$--legs diagrams;
  \item as relations: $\AS$, $\IHX$, $\LE$, $\Hol$, $\OR$, $\LV$, $\LD$, $\EV$ and $\Aut$
   on $4_\tinf$--legs generators and the following
   relations, where $H_1$, $H_2$, $H_3$, $H_4$ are the $\H$--diagrams given in Figure \ref{fig:Haches}:
   \[
   \left\lbrace\begin{array}{l}
                  D_1=D_2\\
                  (a+2)D_1=r(H_3-H_4)\\
                  aG_1+2G_2=rH_1\\
                  G_1+aG_2+G_4=rH_3\\
                  aG_3+2G_4=rH_4\\
                  (a+1)G_2+G_3=rH_2
                 \end{array}\right..
   \]
 \end{itemize}
\end{lemma}

\begin{figure}
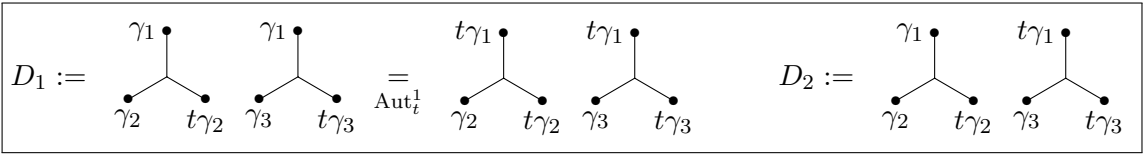
 
  \[\fbox{$
  D_1:=\sixlegsbis{}{}t{}{}t\underset{\Aut^1_t}{=}\sixlegsbis t{}tt{}t
  \hspace{1cm}
  D_2:=\sixlegsbis{}{}tt{}t$}
  \]
  \caption{First family of 6--legs generators} \label{fig6legs1}
 \end{figure}
 \begin{figure}
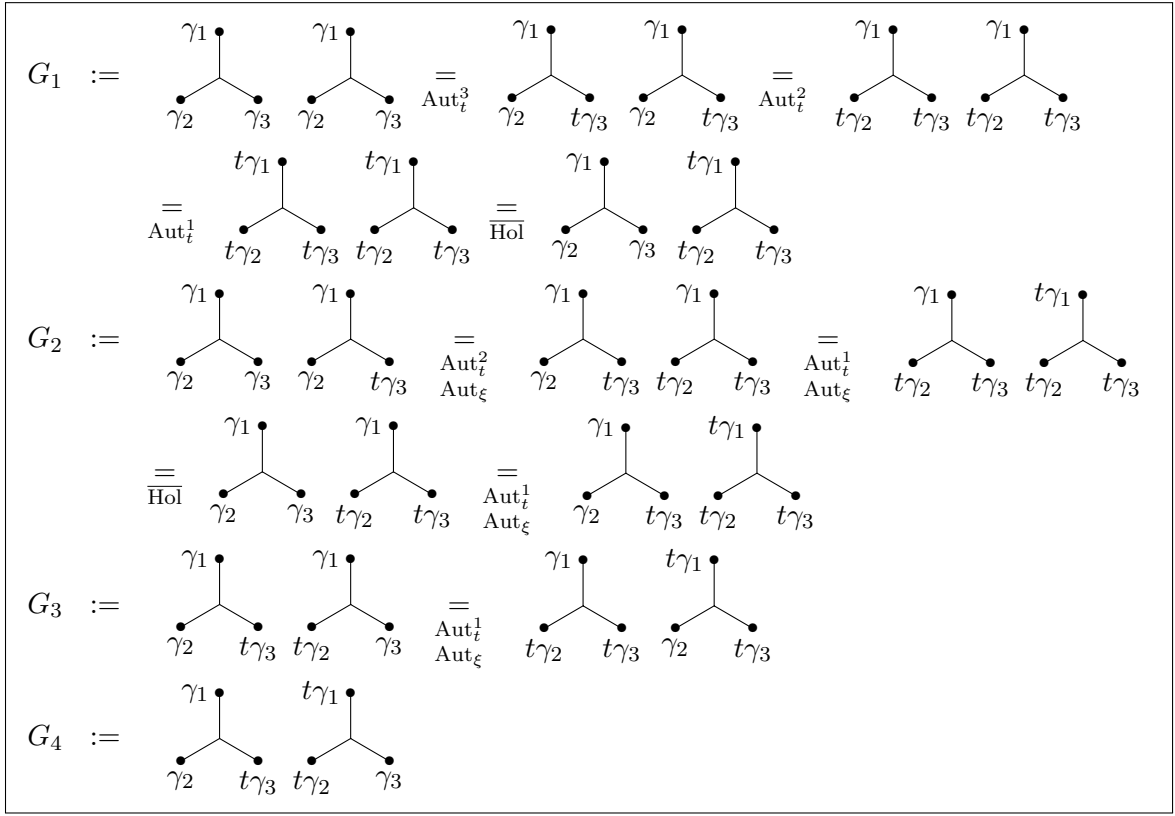

   \[\fbox{$
  \begin{array}{rcl}
  G_1&:=&\sixlegs{}{}{}{}{}{}
    \underset{\Aut^3_t}{=}\sixlegs{}{}{t}{}{}{t}
    \underset{\Aut^2_t}{=}\sixlegs{}{t}{t}{}{t}{t}\\
    &&
    \underset{\Aut^1_t}{=}\sixlegs{t}{t}{t}{t}{t}{t}\underset{\bHol}{=}\sixlegs{}{}{}{t}{t}{t}\\
  G_2&:=&\sixlegs{}{}{}{}{}{t}
          \underset{\begingroup\scriptsize
          \begin{array}{c}\Aut^2_t\\\Aut_\xi\end{array}\endgroup}{=}\sixlegs{}{}{t}{}{t}{t}
          \underset{\begingroup\scriptsize
          \begin{array}{c}\Aut^1_t\\\Aut_\xi\end{array}\endgroup}{=}\sixlegs{}{t}{t}{t}{t}{t}\\
     &&\underset{\bHol}{=}\sixlegs{}{}{}{}{t}{t} 
        \underset{\begingroup\scriptsize
          \begin{array}{c}\Aut^1_t\\\Aut_\xi\end{array}\endgroup}{=}\sixlegs{}{}{t}{t}{t}{t}\\
    G_3&:=&\sixlegs{}{}{t}{}{t}{}
            \underset{\begingroup\scriptsize
          \begin{array}{c}\Aut^1_t\\\Aut_\xi\end{array}\endgroup}{=}\sixlegs{}{t}{t}{t}{}{t}\\
  G_4&:=&\sixlegs{}{}{t}{t}{t}{}
  \end{array}$}\]
 \caption{Second family of 6--legs generators} \label{fig6legs2}
 \end{figure}
\begin{figure}
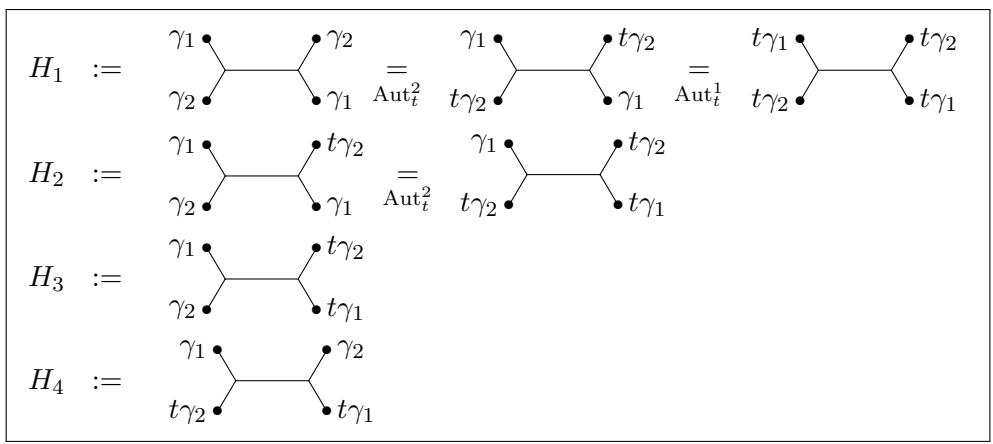

\[\fbox{$ 
 \begin{array}{rcl}
  H_1 &:=& \fourlegs{}{}{}{}\underset{\Aut^2_t}{=}\fourlegs{}t{}t\underset{\Aut^1_t}{=}\fourlegs{t}ttt \\
  H_2 &:=& \fourlegs{}{}{}t\underset{\Aut^2_t}{=}\fourlegs{}ttt \\
  H_3 &:=& \fourlegs{}{}tt \\
  H_4 &:=& \fourlegs{}tt{}
 \end{array}$}
\]
 \caption{Family of 4--legs generators} \label{fig:Haches}
\end{figure}

\begin{proof} 
 Thanks to Lemmas \ref{lemma3step6legs} and \ref{lemmacheckrelations}, we only have to check 
 that the relations $\bHol$ and $\Aut_t$ applied to the
 admissible diagrams of Figures \ref{fig6legs1} and
 \ref{fig6legs2} give exactly the six new relations.
 
 We begin with the first family. Applying $\Aut^2_t$ to
 $D_1$, we obtain:
 \[
 \sixlegsbis{}{}t{}{}t=\sixlegsbis{}t{t^2}{}{}t.
 \]
 By Corollary \ref{cor:DV}, we have:
 \[
 \sixlegsbis{}t{t^2}{}{}t=-a\sixlegsbis{}tt{}{}t-\sixlegsbis{}t{}{}{}t+r\fourlegsLoop.
 \]
 In this equality, the second and fourth diagrams are trivial by $\AS$ and we get $D_1=D_1$. Application of $\Aut^3_t$ to $D_1$ is similar 
 and gives the same result. Now, applying the $\bHol$ relation to $D_1$, we obtain:
 \[
 \sixlegsbis{}{}t{}{}t=\sixlegsbis tt{t^2}{}{}t.
 \]
 Applying Corollary \ref{cor:DV} as previously, we get $D_1=D_2$.  One can check that
 applying $\bHol$ and $\Aut_t$ to the second form of $D_1$ does not give any additional relation. 

 We now have to apply the same relations to $D_2$. Applying $\Aut^1_t$ to $D_2$ gives:
 \[
 \sixlegsbis{}{}tt{}t=\sixlegsbis t{}t{t^2}{}t.
 \]
 Once again we use Corollary \ref{cor:DV} to get:
 \[
 \sixlegsbis t{}t{t^2}{}t=-a\sixlegsbis t{}tt{}t-\sixlegsbis t{}t{}{}t+r\fourlegsBas,
 \]
 and finally:
 \[
 D_1=\frac{r}{a+2}\fourlegsop{\gamma_2}{t\gamma_2}{\gamma_3}{t\gamma_3}{0.12}.
 \]
 One can check that applying the other $\Aut_t$ or the $\bHol$ relations to $D_2$ does not give any additional relation. 

 We turn to the second family of 6--legs generators. 
 Applying $\Aut^3_t$ to $G_2$ gives:
 \[
 \sixlegs{}{}{}{}{}t=\sixlegs{}{}t{}{}{t^2},
 \]
 and by Corollary \ref{cor:DV}, we have:
 \[
 \sixlegs{}{}t{}{}{t^2}=-\,a\sixlegs{}{}t{}{}t-\sixlegs{}{}t{}{}{}+r\fourlegsRevers,
 \]
 so we get the relation:
 \[
 aG_1+2G_2=r\fourlegs{}{}{}{}.
 \]
 Application of $\bHol$ gives:
 \[
 \sixlegs{}{}t{}{}t=\sixlegs{}{}ttt{t^2},
 \]
 which, developed with Corollary \ref{cor:DV}, gives:
 \[
 G_1+aG_2+G_4=r\fourlegs{}{}tt.
 \]
 By $\Aut^1_t$ and $\Aut^2_t$ respectively, we get:
 \[
 \sixlegs{}{}ttt{}=\sixlegs{t}{}t{t^2}t{}
 \]
 and
 \[
 \sixlegs{}{}t{}t{}=\sixlegs{}tt{}{t^2}{},
 \]
 which, using Corollary \ref{cor:DV}, provides respectively:
 \[
 aG_3+2G_4=r\fourlegs{}tt{}\hspace{.8cm}\textrm{and}\hspace{.8cm}(a+1)G_2+G_3=r\fourlegs{t}{}{}{}.
 \]
  One can check that the other relations $\Aut_t$ and $\bHol$ applied
  to the different given forms of the $G_i$'s do not provide further
  relations.
\end{proof}

\begin{corollary} \label{corlaststep6legs}
 The space $\A_2\big(\BlModT\big)$ admits the presentation with:
 \begin{itemize}
  \item as generators: the diagram $G_1$ given in Figure \ref{figbasis} and $4_\tinf$--legs diagrams;
  \item as relations: $\AS$, $\IHX$, $\LE$, $\Hol$, $\OR$, $\LV$, $\LD$, $\EV$ and $\Aut$ on $4_\tinf$--legs generators 
   and the following relation between $G_1$ and the $\H$--diagrams given in Figure \ref{fig:Haches}:
   \begin{equation} \tag{$R_6$}
    (1-a)(a+2)^2G_1=4H_3+2aH_2-2H_4-a(a+3)H_1.
    \label{eqR6}
   \end{equation}
 \end{itemize}
\end{corollary}

Now, we turn our attention to 4--legs generators.

\begin{lemma} \label{lemma4step4legs}
 The space $\widehat{\A}^{(4)}_2\big(\BlModT\big)$ admits the presentation with:
\begin{itemize}
  \item as generators: the $\H$--diagrams $H_1$, $H_2$, $H_3$, $H_4$ given in Figure \ref{fig:Haches} and $2_\tinf$--legs diagrams;
  \item as relations: $\AS$, $\IHX$, $\LE$, $\Hol$, $\OR$, $\LV$, $\LD$, $\EV$ and $\Aut$ on $2_\tinf$--legs generators and the following two relations:
  \begin{eqnarray*}
  aH_1+2H_2&=&-r\twolegs{}{}\\
  aH_2+H_3+H_4&=&-r\twolegs{}t.
 \end{eqnarray*}
 \end{itemize}
\end{lemma}
\begin{proof} 
 Thanks to Lemmas \ref{lemma3step4legs} and \ref{lemmacheckrelations}, we only have to check that $\Aut_t$ applied to the diagrams of Figure \ref{fig:Haches} 
 provides exactly the above two relations. This is straightforward. 
\end{proof}

\begin{corollary} \label{corlaststep4legs}
 The space $\widehat{\A}^{(4)}_2\big(\BlModT\big)$ admits the presentation with:
 \begin{itemize}
  \item as generators: the $\H$--diagrams $H_1$ and $H_3$ given in Figure \ref{fig:Haches} and $2_\tinf$--legs diagrams;
  \item as relations: $\AS$, $\IHX$, $\LE$, $\Hol$, $\OR$,
    $\LV$, $\LD$, $\EV$ and $\Aut$ on $2_\tinf$--legs generators.
 \end{itemize}
\end{corollary}

\begin{proof}[Proof of Proposition \ref{propbasis}]
 Thanks to Corollaries \ref{corlaststep6legs} and
 \ref{corlaststep4legs}, $\A_2\big(\BlModT\big)$ has a presentation
 given by the generators $G_1$, $H_1$, $H_3$ and 
 all $2_\tinf$--legs diagrams, and the relation \eqref{eqR6} and all usual relations on $2_\tinf$--legs diagrams. Using \eqref{eqR6} to write $H_3$ in terms 
 of the other generators, we obtain a presentation with, as
 generators, $G_1$, $H_1$ and $2_\tinf$--legs diagrams and, as
 relations, the usual relations on $2_\tinf$--legs diagrams.
 This concludes the first two points of the proposition.
  The third point is given by Corollary \ref{corlaststep4legs}.
 
 If $a\neq1$, in the presentation of $\A_2\big(\BlModT\big)$ given in Corollary \ref{corlaststep6legs}, one can remove the generator $G_1$ and the relation \eqref{eqR6}. 
 This implies the fourth point of the proposition. 

 If $a=1$, in the presentation of $\A_2\big(\BlModT\big)$ given
 in Corollary \ref{corlaststep6legs}, $G_1$ is not subject to any
 relation. On the other hand, compared with Lemma
 \ref{lemma4step4legs}, \eqref{eqR6} provides then a third relation
 between the $H_i$'s which holds in
 $\A_2\big(\BlModT\big)$ but not in $\widehat{\A}^{(4)}_2\big(\BlModT\big)$. This new
 relation can be used to show that $H_1$ and $H_3$ are equal up to diagrams with fewer legs. This concludes the fifth point of the proposition.
\end{proof}

\subsection{On the maps $\iota_2$}

The main goal of this section is to determine the injectivity and
surjectivity status of the maps $\iota_2^1:\A_2\BlMod\to\A_2\left(\BlModT\right)$ and $\iota_2^2:\A_2\big(\BlModD\big)\to\A_2\left(\BlModT\right)$ when $\Al$
is of $\Q$--dimension two and cyclic.
It is a direct consequence of Corollary \ref{corinjectivity} and
Proposition \ref{propbasis} that:
\begin{proposition}
 If $\BlMod$ is a cyclic Blanchfield module of $\Q$--dimension $2$ with annihilator different from $t+1+t^{-1}$, then the maps $\iota_2^1$ and $\iota_2^2$ 
 are injective. 
\end{proposition}

It remains to deal with injectivity when $\delta=t+1+t^{-1}$ and to determine the surjectivity status of the maps $\iota_2$. We start with $\iota_2^1$.

\begin{proposition} \label{proponecopy}
 Let $\BlMod$ be a cyclic Blanchfield module of $\Q$-dimension two. Then the map $\iota_2^1$ is injective but not surjective.
\end{proposition}
\begin{proof}
 Thanks to the first point of Proposition \ref{propbasis} and Corollary \ref{coridealpres}, the map $\iota_2^1$ induces an isomorphism 
 from $\A_2^{(2)}\BlMod$ to $\A_2^{(2)}\left(\BlModT\right)$. Hence we can work with the map $\overline{\iota}_2^1$ 
 induced by $\iota_2^1$ on the quotients $\fract{\A_2}{\A_2^{(2)}}$. 
 
 It is easy to check that $\fract{\A_2\BlMod}{\A_2^{(2)}\BlMod}$ is generated by the following $\H$--diagram: 
 \[
 G=\fourlegsop{\gamma}{t\gamma}{\gamma}{t\gamma}{0.15}.
 \]
 By \cite[Proposition 7.10]{M7}, $\overline{\iota}_2^1(G)$ is half the sum of all diagrams obtained from $G$ by replacing two $\gamma$'s by $\gamma_1$ 
 and the other two by $\gamma_2$. Thanks to $\Aut_\xi$, this gives:
 \[
 \overline{\iota}_2^1(G)=\fourlegsop{\gamma_1}{t\gamma_1}{\gamma_2}{t\gamma_2}{0.12}+\fourlegsop{\gamma_1}{t\gamma_2}{\gamma_1}{t\gamma_2}{0.12}
                   +\fourlegsop{\gamma_1}{t\gamma_2}{\gamma_2}{t\gamma_1}{0.12}.
 \]
 Applying an $\IHX$ relation to the first diagram, $\Aut^2_t$ to the second one and various $\AS$ relations, it can be reformulated into:
 \[
 \overline{\iota}_2^1(G)=\fourlegs{}{}{}{}+\fourlegs{}{}tt-2\fourlegs{}tt{}.
 \]
 Using Relation \eqref{eqR6} and the relations of Lemma \ref{lemma4step4legs}, we finally obtain:
 \[
 \overline{\iota}_2^1(G)=\frac12(1-a)(a+2)^2G_1+\frac12(a+1)(a+2)H_1,
 \]
 up to $2_\tinf$--legs diagrams. It follows by the second point of Proposition \ref{propbasis} that $\overline{\iota}_2^1$ is injective 
 but not surjective.
\end{proof}

We now deal with the map $\iota_2^2$. For that, we have to study the structure of $\A_2\big(\BlModD\big)$. 
The next lemma describes the elements of $\Aut\big(\BlModD\big)$ for a cyclic Blanchfield module $\BlMod$ with irreducible annihilator. 
For $P\in\Qt$, set $\bar{P}(t)=P(t^{-1})$.
\begin{lemma} \label{lemmaautA2}
If $\delta$ is irreducible in $\Qt$, then 
 the group $\Aut\big(\BlModD\big)$ is generated by the automorphisms
 \[
\chi_P:\left\{\begin{array}{ccl}
  \gamma_1&\mapsto& P\gamma_1\\
  \gamma_2&\mapsto& \gamma_2
  \end{array}\right.
\]
 for $P\in\Qt$ such that $P\bar{P}=1\ mod\ \delta$ and 
 \[
  \lambda_{P,Q}:
  \left\{\begin{array}{ccl}
  \gamma_1&\mapsto& P\gamma_1+Q\gamma_2\\
  \gamma_2&\mapsto& \bar{Q}\gamma_1-\bar{P}\gamma_2
  \end{array}\right.
 \]
 for $P,Q\in\Qt$ such that $P\bar{P}+Q\bar{Q}=1\ mod\ \delta$.
\end{lemma}
\begin{proof}
 In the whole proof, polynomials are considered in $\fract{\Qt}{(\delta)}$. For $P\in\Qt$ such that $P\bar{P}=1$, 
 define
 \[
\chi_P':\left\{\begin{array}{ccl}
  \gamma_1&\mapsto& \gamma_1\\
  \gamma_2&\mapsto& P\gamma_2
  \end{array}\right.,
\]
and note that $\chi_P'=\lambda_{0,1}\circ\chi_P\circ\lambda_{0,1}$. 
 Let $\zeta\in\Aut\big(\BlModD\big)$ and write
 \[
\zeta:
  \left\{\begin{array}{ccl}
  \gamma_1&\mapsto& P\gamma_1+Q\gamma_2\\
  \gamma_2&\mapsto& R\gamma_1+S\gamma_2
  \end{array}\right..
\]
Since $\zeta$ must preserve $\bl$, we have $P\bar{P}+Q\bar{Q}=1$, $R\bar{R}+S\bar{S}=1$ and $P\bar{R}+Q\bar{S}=0$. 
 If $Q=0$, then $P\bar{R}=0$, so that $R=0$ and $\zeta=\chi_P\circ\chi_S'$. 
 If $Q\neq0$, then $S=-\bar{Q}^{-1}\bar{P}R$, so that 
\[
1=R\bar{R}+S\bar{S}=R\bar{R}(Q\bar{Q})^{-1}(Q\bar{Q}+P\bar{P})=R\bar{R}(Q\bar{Q})^{-1}.
\]
 Finally $\bar{Q}^{-1}R\overline{\bar{Q}^{-1}R}=1$ and $\zeta=\lambda_{P,Q}\circ\chi_{\bar{Q}^{-1}R}'$.
\end{proof}

We denote by $\Aut_\chi$ and $\Aut_\lambda$ the subfamilies of $\Aut$ relations obtained by the action of the automorphisms $\chi_P$ and $\lambda_{P,Q}$
respectively.

\begin{proposition}\label{prop:PreTwocopies} If $(\Al,\bl)$ is a cyclic Blanchfield module of $\Q$--dimension $2$, then:
  \begin{enumerate}
  \item\label{item:Lem1} $\A_2\big(\BlModD\big)=\A_2^{(4)}\big(\BlModD\big)$,
  \item\label{item:Lem2} $\A_2^{(2)}\big(\BlModD\big)=\widehat{\A}_2^{(2)}\big(\BlModD\big)$,
  \item\label{item:Lem3} $\A_2^{(4)}\big(\BlModD\big)\cong\widehat{\A}_2^{(4)}\big(\BlModD\big)$,
  \item\label{item:Lem4}  $\fract{\A_2^{(4)}\big(\BlModD\big)}{\A_2^{(2)}\big(\BlModD\big)}\cong
     \fract{\widehat{\A}_2^{(4)}\big(\BlModT\big)}{\widehat{\A}_2^{(2)}\big(\BlModT\big)}$; in particular, this quotient has $\Q$--dimension $2$.
  \end{enumerate}
\end{proposition}
\begin{proof}
 It is easy to see that $\fract{\A_2\big(\BlModD\big)}{\A_2^{(4)}\big(\BlModD\big)}$ is generated by the diagrams $\Gamma_1$ and $\Gamma_2$ of Figure \ref{figGamma1}. 
 \begin{figure} 
  \[\fbox{$
  \Gamma_1=\sixlegsop{\gamma_1}{\gamma_2}{t\gamma_2}{\gamma_1}{\gamma_2}{t\gamma_2} \quad 
  \Gamma_2=\sixlegsop{\gamma_1}{\gamma_2}{t\gamma_2}{t\gamma_1}{\gamma_2}{t\gamma_2} \quad 
  \Gamma_3=\sixlegsop{t\gamma_1}{\gamma_2}{t\gamma_2}{t\gamma_1}{\gamma_2}{t\gamma_2}$}
  \]
  \caption{Some admissible $\YY$--diagrams} \label{figGamma1}
 \end{figure}
 Application of an $\bHol$ relation to $\Gamma_1$ followed by a use of Corollary \ref{cor:DV} gives:
 \[
 \Gamma_1-\Gamma_2=r\fourlegs{}{}tt+r\fourlegsop{t^2\gamma_2}{t\gamma_1}{t\gamma_1}{t\gamma_2}{0.12}=r\fourlegs{}{}tt-r\fourlegs{}t{}{},
 \]
 where the second equality comes from a $\bHol$ and an $\AS$ relations on the second diagram. Application of $\Aut^1_t$ to $\Gamma_2$ followed by a use of Corollary \ref{cor:DV} and again of an $\Aut^1_t$ relation gives:
 \[
 a\Gamma_1+2\Gamma_2=r\fourlegsD{}t{}t=r\fourlegsU{}t{}t,
 \]
 where the second equality comes from an $\Aut_\xi$ relation.
 Since $a\neq-2$, it follows that both $\Gamma_1$ and $\Gamma_2$ can be expressed in term of 4--legs generators. 
 Hence $\A_2\big(\BlModD\big)=\A_2^{(4)}\big(\BlModD\big)$, that is the first point of the proposition. 
 
 The second point follows from Proposition \ref{propbasis} (1) and Corollaries \ref{corpresentation} and \ref{coridealpres}. 
 Note that we have:
 $$\A_2^{(2)}\big(\BlModD\big)\cong\widehat{\A}_2^{(2)}\big(\BlModD\big)\cong\widehat{\A}_2^{(2)}\big(\BlModT\big)\cong\A_2^{(2)}\big(\BlModT\big).$$
 Hence, to prove the third point, we can work on the quotients $\fract{\A_2}{\A_2^{(2)}}$ and $\fract{\widehat{\A}_2}{\widehat{\A}_2^{(2)}}$.

 If $a\neq1$, the third point is given by Corollary \ref{corinjectivity} thanks to the first and fourth points of Proposition \ref{propbasis}. 
 Assume $a=1$. The diagrams $\Gamma_i$ for $i=1,\dots,6$ represented in Figures~\ref{figGamma1} and~\ref{figGamma2} form a minimal
 essential set $\E$ of admissible $\YY$--diagrams. 
 Thanks to Lemmas \ref{lemma3step6legs}, \ref{lemmacheckrelations} and \ref{lemmaautA2},
 we only need to consider $\bHol^\E$, $\Aut_\chi^\E$ and $\Aut_\lambda^\E$.
 \begin{figure}
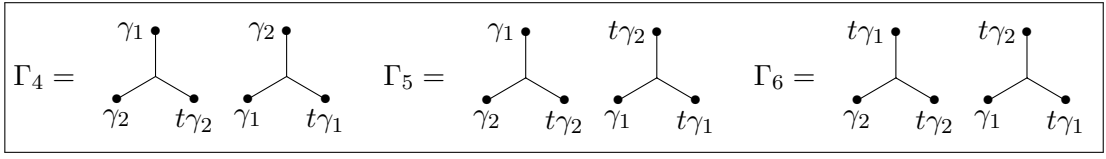
 
  \[\fbox{$
  \Gamma_4=\sixlegsop{\gamma_1}{\gamma_2}{t\gamma_2}{\gamma_2}{\gamma_1}{t\gamma_1} \quad 
  \Gamma_5=\sixlegsop{\gamma_1}{\gamma_2}{t\gamma_2}{t\gamma_2}{\gamma_1}{t\gamma_1} \quad 
  \Gamma_6=\sixlegsop{t\gamma_1}{\gamma_2}{t\gamma_2}{t\gamma_2}{\gamma_1}{t\gamma_1}$}
  \]
  \caption{Some trivial admissible $\YY$--diagrams} \label{figGamma2}
 \end{figure}
 The $\bHol$ and $\Aut_\chi$ relations applied to $\Gamma_i$ with $i>3$ obvioulsy give trivial relations; check that the relations $\Aut_\lambda$ applied 
 to these diagrams also give trivial relations thanks to cancellations in the decomposition.
 
 The $\bHol$ relation applied to $\Gamma_1$ or $\Gamma_2$ recovers the above two relations. Up to these two relations, $\bHol$ applied to $\Gamma_3$ 
 gives a trivial relation up to $2_\tinf$--legs diagrams.
 
 It remains to write the $\Aut^\E$ relations corresponding to the $\Gamma_i$'s with $i\leq3$. A relation $\Aut_\chi$ with an automorphism $\chi_P$ 
 applied to $\Gamma_3$ is recovered from the relation $\Aut_\chi$ with $\chi_{tP}$ applied to $\Gamma_1$. The relations $\Aut_\chi$ applied to $\Gamma_1$ 
 and $\Gamma_2$ can be written by hand. However, the relations $\Aut_\lambda$ imply wild computations which required the help of a computer. 
 The program given in Appendix \ref{sec:Program} checks that a relation $\Aut_\lambda$ applied on $\Gamma_i$ for $i=1,2,3$ can be recovered 
 from the above two relations and usual relations on $4_\tinf$--legs generators. This concludes the third point of the proposition. 

 We have seen that 
 $\fract{\A_2^{(4)}\big(\BlModD\big)}{\A_2^{(2)}\big(\BlModD\big)}\cong\fract{\widehat{\A}_2^{(4)}\big(\BlModD\big)}{\widehat{\A}_2^{(2)}\big(\BlModD\big)}$. 
 By Corollary \ref{corpresentation}, we have $\fract{\widehat{\A}_2^{(4)}\big(\BlModD\big)}{\widehat{\A}_2^{(2)}\big(\BlModD\big)}
 \cong\fract{\widehat{\A}_2^{(4)}\big(\BlModT\big)}{\widehat{\A}_2^{(2)}\big(\BlModT\big)}$. This gives the isomorphism of the fourth point. 
 The dimension of the quotient is given by the third point of Proposition \ref{propbasis}. 
\end{proof}

\begin{proposition} \label{proptwocopies}
 Let $\BlMod$ be a cyclic Blanchfield module of $\Q$--dimension two,
 with annihilator $\delta$. Then the map $\iota_2^2:\A_2\big(\BlModD\big)\to\A_2\left(\BlModT\right)$:
 \begin{itemize}
  \item is an isomorphism if $\delta\neq t+1+t^{-1}$;
  \item has a non trivial kernel generated by the combination of $\H$--diagrams
   \[
   2\fourlegs{}{}{}{}+\fourlegs{}tt{}-2\fourlegs{}{}tt-\fourlegs{}{}{}t
   \]
   if $\delta=t+1+t^{-1}$.
 \end{itemize}
\end{proposition}
\begin{proof}
 First assume $\delta\neq t+1+t^{-1}$. The fourth point of Proposition \ref{propbasis} and Corollary \ref{coridealpres} imply that $\iota_2^2$ induces 
 an isomorphism from $\A_2^{(4)}\big(\BlModD\big)$ to $\A_2^{(4)}\big(\BlModT\big)$. This proves the first point since 
 $\A_2\big(\BlModD\big)=\A_2^{(4)}\big(\BlModD\big)$ by Proposition
 \ref{prop:PreTwocopies} and
 $\A_2\big(\BlModT\big)=\A_2^{(4)}\big(\BlModT\big)$ by the fourth
 point of Proposition \ref{propbasis}. 
 
 Now assume that $\delta=t+1+t^{-1}$. The second point of Proposition \ref{prop:PreTwocopies} asserts that
 $\A_2^{(2)}\big(\BlModD\big)\cong\widehat{\A}_2^{(2)}\big(\BlModD\big)$. Moreover,
 $\A_2^{(2)}\big(\BlModT\big)\cong\widehat{\A}_2^{(2)}\big(\BlModT\big)$
 by the first point of Proposition \ref{propbasis}. Hence it follows
 from Corollary \ref{corpresentation} 
 that $\iota_2^2$ is an isomorphism at the $\A_2^{(2)}$--level. 
 By the first point of Proposition \ref{prop:PreTwocopies}, the quotient 
 $\fract{\A_2\big(\BlModD\big)}{\A_2^{(2)}\big(\BlModD\big)}$ is equal to
 $\fract{\A^{(4)}_2\big(\BlModD\big)}{\A_2^{(2)}\big(\BlModD\big)}$,
 so its image by $\iota_2^2$ is included in
 $\fract{\A_2^{(4)}\left(\BlModT\right)}{\A_2^{(2)}\left(\BlModT\right)}$. Now,
 the $\H$--diagram $H_1$ of Figure \ref{fig:Haches} is clearly in the image of $\iota_2^2$. 
 Finally, by Proposition \ref{propbasis} (\ref{item:Prop5}.\ref{item:Prop51}) and
 Proposition \ref{prop:PreTwocopies} (\ref{item:Lem4}), the kernel of $\iota_2^2$ has dimension~1.
 
 More precisely, thanks to Relation \eqref{eqR6}, the image through $\iota_2^2$ of
 \[
 D=2\fourlegs{}{}{}{}+\fourlegs{}tt{}-2\fourlegs{}{}tt-\fourlegs{}{}{}t
 \]
 is zero. In the quotient $\fract{\widehat{\A}_2^{(4)}\big(\BlModT\big)}{\widehat{\A}_2^{(2)}\big(\BlModT\big)}$, $D$ is equal to $3(H_1-H_3)$, which is non zero 
 by Proposition \ref{propbasis} (\ref{item:Prop3}). Moreover, 
 $\fract{\widehat{\A}_2^{(4)}\big(\BlModT\big)}{\widehat{\A}_2^{(2)}\big(\BlModT\big)}\cong\fract{\A_2\big(\BlModD\big)}{\A^{(2)}_2\big(\BlModD\big)}$ 
 by Proposition \ref{prop:PreTwocopies} (\ref{item:Lem1},\ref{item:Lem4}).
 It follows that $D$ is non trivial in $\A_2\big(\BlModD\big)$. 
\end{proof}

\section{Case when $\mathbf{\Al}$ is of $\Q$--dimension two and non cyclic} \label{secnoncyclic}

In this section, we assume that $\BlMod$ is a non cyclic Blanchfield
module of $\Q$--dimension two. As mentioned at the beginning of Section \ref{secpreliminaries}, it implies that $\Al$ is the direct sum 
of two $\Qt$-modules of order $t+1$. Hence we can write:
\[
\Al=\frac{\Qt}{(t+1)}\gamma\oplus\frac{\Qt}{(t+1)}\eta.
\]
Moreover, it follows from $\bl$ being hermitian and non-degenerate
that, up to rescaling $\eta$, $\bl(\gamma,\gamma)=\bl(\eta,\eta)=0$ and $\bl(\gamma,\eta)=\frac1{t+1}$. Throughout the
section, we consider $\{\gamma,\eta\}$ as the basis $\omega$
for $\Al$ and we set $f(\gamma,\gamma)=f(\eta,\eta)=0$, $f(\gamma,\eta)=\frac1{t+1}$ and $f(\eta,\gamma)=\frac t{t+1}$. Accordingly,
set $\gamma_i=\xi_i(\gamma)$ and $\eta_i=\xi_i(\eta)$, for $i=1,2,3$.

\begin{lemma}\label{lem:AutSum}
 The automorphism group $Aut\BlMod$ is generated by the following automorphisms:
 \[\mu_x:\left\lbrace\begin{array}{lll}
                      \gamma & \mapsto & x\gamma \\
                      \eta & \mapsto & x^{-1}\eta
                     \end{array}\right.
   \qquad \nu:\left\lbrace\begin{array}{lll}
                            \gamma & \mapsto & \eta \\
                            \eta & \mapsto & -\gamma
                           \end{array}\right.
   \qquad \rho_y:\left\lbrace\begin{array}{lll}
                                  \gamma & \mapsto & \gamma+y\eta \\
                                  \eta & \mapsto & \eta 
                                 \end{array}\right.
 \]
where $x$ runs over $\Q\setminus\{0,\pm1\}$ and $y$ over $\Q\setminus\{0\}$.
\end{lemma}
\begin{proof}
 Any automorphism $\zeta$ of $\BlMod$ is given by
 \[
 \zeta:\left\lbrace\begin{array}{lll}
                    \gamma & \mapsto & x\gamma+y\eta \\
                    \eta & \mapsto & z\gamma+w\eta
                   \end{array}\right.
 \]
 with $x,y,z,w$ in $\Q$. Since $\zeta$ preserves the Blanchfield pairing $\bl$, we have $xw-yz=1$. 
 If $z=0$, then $xw=1$ and $\zeta=\rho_{yx^{-1}}\circ\mu_x$. If $w=0$, then $yz=-1$ and $\zeta=\nu\circ\rho_{-xy^{-1}}\circ\mu_y$. 
 Finally, if $zw\neq0$, then $\zeta=\mu_{w^{-1}}\circ\nu\circ\rho_{-zw}\circ\nu^{-1}\circ\rho_{yw^{-1}}$. 
\end{proof}

We denote by $\Aut_\mu$, $\Aut_\nu$ and $\Aut_\rho$ the subfamilies of $\Aut$ relations obtained by the
action of the automorphisms given by $\mu_x$, $\nu$ and $\rho_y$ respectively on one copy of $\Al$ and identity on the others.

\begin{proposition} \label{propprest+1}
  If $(\Al,\bl)$ is a non cyclic Blanchfield module of $\Q$--dimension two, then:
  \begin{enumerate}
  \item\label{item:Propo1}
    $\A_2^{(2)}\left(\BlModT\right)\cong\widehat{\A}_2^{(2)}\left(\BlModT\right)$;
  \item\label{item:Propo2} $\A_2\left(\BlModT\right)=\A_2^{(4)}\left(\BlModT\right)\cong\widehat{\A}_2^{(4)}\left(\BlModT\right)$;
  \item\label{item:Propo3}  $\fract{\A_2\left(\BlModT\right)}{\A_2^{(2)}\left(\BlModT\right)}$ is
    freely generated by the admissible $\H$--diagram
    \[
    \fourlegsop{\gamma_1}{\gamma_2}{\eta_1}{\eta_2}{0.12}.
    \]
  \end{enumerate}
\end{proposition}
\begin{proof}
 We start with the presentation given by Lemma \ref{lemma3step6legs} to deal with 6--legs generators. 
 Let $D$ be an admissible $\YY$--diagram with two legs $v$ and $w$ labelled by the same $\gamma_i$ or the same $\eta_i$. 
 Application of any $\Aut_\mu$ relation shows that the diagram $D$ is trivial. Application of an $\Aut_\nu$, $\Aut_\xi$ or
 $\bHol$ relation to $D$ gives a trivial relation in $\Aut_\nu^\omega$, $\Aut_\xi^\omega$ or $\bHol^\omega$. Application 
 of an $\Aut_\rho$ relation to $D$ gives in $\Aut_\rho^\omega$ the relation of $\Aut_\nu^\omega$ obtained by applying 
 $\Aut_\nu$ to the diagram $D'$ obtained from $D$ by changing the labels of $v$ and $w$ to $\gamma_i$ and $\eta_i$ respectively 
 and the linking $f_{vw}$ to $\frac{1}{t+1}$. Hence we can remove from the generators the admissible $\YY$--diagrams with a common label 
 on two distinct legs without adding any relation. Then, using Lemma \ref{lemmacheckrelations}, it is easily seen that
 one can restrict the $6$--legs generators to the admissible $\YY$--diagrams:
 \[
 Y_1=\sixlegsop{\gamma_1}{\gamma_2}{\eta_2}{\eta_1}{\gamma_3}{\eta_3}
 \hspace{1cm}\textrm{and}\hspace{1cm}
 Y_2=\sixlegsop{\gamma_1}{\gamma_2}{\gamma_3}{\eta_1}{\eta_2}{\eta_3}.
 \]
 On these generators, $\Aut_\mu$ and $\Aut_\xi$ act trivially, so we are left
 with checking the relations coming from $\bHol$ and $\Aut_\nu$ relations. Note however that applying these relations may change the prescribed the rational fractions on pairs of vertices, so that use of an $\LD$ relation may be needed to correct them. For instance, application of $\Aut_\nu$, regarding $\Al_1$, on $Y_1$ gives
\[
Y_1=\sixlegsop{-\eta_1}{\gamma_2}{\eta_2}{\gamma_1}{\gamma_3}{\eta_3}+\fourlegsop{\gamma_2}{\eta_2}{\gamma_3}{\eta_3}{0.12}.
\]
The prescribed rational fraction between the top vertices of $Y_1$ is indeed $\frac{1}{1+t}$, whereas the one of the $6$--legs term on the right is $f(-\eta,\gamma)=\frac{-t}{1+t}=\frac{1}{1+t}-1$; use of an $\LD$ relation is hence needed and produces the $4$--legs term. Then applications of $\LV$ and $\Aut_\xi$ relations lead to
 \[
 2Y_1=\fourlegsop{\gamma_2}{\eta_2}{\gamma_3}{\eta_3}{0.12}.
 \]
 Similarly, application of $\bHol$ to $Y_1$ gives
\[
Y_1=\sixlegsop{t\gamma_1}{t\gamma_2}{t\eta_2}{\eta_1}{\gamma_3}{\eta_3}=-\sixlegsop{\gamma_1}{\gamma_2}{\eta_2}{\eta_1}{\gamma_3}{\eta_3}+\fourlegsop{\gamma_2}{\eta_2}{\gamma_3}{\eta_3}{0.12}.
\]
Here, the second equality is due to the fact that $tx=-x$ for any $x\in\Al$. The rational fraction on the top pair of vertices has to be corrected so that it corresponds to the prescribed one; this produces the $4$--legs term. Once again, we get
 \[
 2Y_1=\fourlegsop{\gamma_2}{\eta_2}{\gamma_3}{\eta_3}{0.12}.
 \]
Applications of $\Aut_\nu$ on $\Al_2$ and $\Al_3$ give trivial relations. 
 On $Y_2$, the only relations that do act non trivially are $\bHol$ and
 $\Aut_\nu$ applied simultaneously on the three $\Al_i$; both give:
 \[
 2Y_2=3\fourlegsop{\gamma_1}{\gamma_2}{\eta_1}{\eta_2}{0.12}
      +\twolegsop{\gamma_1}{\eta_1}+\zerolegs.
 \]
 Finally, we can remove all 6--legs generators without adding any
 relation. This proves the second assertion. 
 
 We turn to the study of the 4--legs generators. Thanks to Lemmas \ref{lemma3step4legs} and \ref{lemmacheckrelations} and removing as previously 
 generators with a common label on two distinct legs, we are led to the diagrams:
 \[
 X_1=\fourlegsop{\gamma_1}{\gamma_2}{\eta_1}{\eta_2}{0.12}\quad\textrm{and}\quad X_2=\fourlegsop{\gamma_1}{\eta_2}{\eta_1}{\gamma_2}{0.12}
 \]
 on which we have to check the effect of the $\Aut_\nu$ relations. Applying $\Aut_\nu$ on $\Al_1$ or $\Al_2$ to $X_1$ or $X_2$ always gives:
 \[
 X_1+X_2=-\twolegsop{\gamma_1}{\eta_1}.
 \]
 Since no more relation arises from the 4--legs generators, this proves
 the first and third assertions.
\end{proof}

\begin{proposition} \label{propt+1}
 Let $\BlMod$ be a non cyclic Blanchfield module of $\Q$--dimension
 two. Then the maps $\iota_2^1:\A_2\BlMod\to\A_2\left(\BlModT\right)$ and $\iota_2^2:\A_2\big(\BlModD\big)\to\A_2\left(\BlModT\right)$ are injective. 
 Moreover, $\iota_2^2$ is surjective, while $\iota_2^1$ is not. 
\end{proposition}
\begin{proof}
 It is easily seen that $\A_2\BlMod$ is generated by admissible diagrams. 
 Such a diagram with at least four legs has necessarily two legs
 labelled by $\gamma$ or two legs labelled by $\eta$; the relation
 $\Aut_\mu$ implies that it is trivial. 
 It follows that $\A_2\BlMod=\A_2^{(2)}\BlMod$. 
 Hence, by Proposition \ref{propprest+1} and Corollary~\ref{corinjectivity}, $\iota_2^1$ is injective but not surjective. 

 Similarly, we have
 $\A_2\big(\BlModD\big)=\A_2^{(4)}\big(\BlModD\big)$ and it follows
 from the second point of Proposition \ref{propprest+1}
 and Corollary \ref{coridealpres} that $\iota_2^2$ is an isomorphism. 
\end{proof}

\begin{appendix}
  \section{Programs} \label{sec:Program}

Let $\BlMod$ be a cyclic Blanchfield module with annihilator $\delta=t+1+t^{-1}$. 
Let $\gamma$ be a generator of $\Al$. As recalled at the beginning of
Section \ref{secbasis}, $\bl(\gamma,\gamma)=\frac{r}{\delta}\ mod\
\Qt$ with $r\in\Q^*$. We set $\gamma_i=\xi_i(\gamma)$ for
$i=1,2$. A $\Q$--basis of $\Al^{\oplus 2}$ is given by the $t^\e\gamma_i$ with $\e=0,1$ and $i=1,2$.

This appendix aims at determining the relations induced on 
$\fract{\A_2\big(\BlModD\big)}{\A^{(2)}_2\big(\BlModD\big)}$ by applying the $\Aut_\lambda$ relations to the diagrams $\Gamma_i$ of Figure \ref{figGamma1}.
Set
 \[
  \lambda_{a,b,c,d}:
  \left\{\begin{array}{ccl}
  \gamma_1&\mapsto& (at+b)\gamma_1+(ct+d)\gamma_2\\
  \gamma_2&\mapsto& (ct^{-1}+d)\gamma_1-(at^{-1}+b)\gamma_2
  \end{array}\right.
 \]
for $a,b,c,d\in\Q$ such that $a^2+b^2+c^2+d^2=1+ab+cd$.  
We wrote three programs in
OCaml\footnote{avalaible at
  \url{http://www.i2m.univ-amu.fr/~audoux/Reduc_Gamma#.ml} with \url{#}\,$=1,2,3$.} which
compute the reductions of $\lambda_{a,b,c,d}.\Gamma_1$, 
$\lambda_{a,b,c,d}.\Gamma_2$ and $\lambda_{a,b,c,d}.\Gamma_3$. Here, $a$, $b$, $c$ and $d$
are considered as parameters and all the computations are made in
\[
\Q_{a,b,c,d}:=\fract{\Q[a,b,c,d]}{a^2+b^2+c^2+d^2-ab-cd-1}.
\]
Note that every element in $\Q_{a,b,c,d}$ has a unique representative in $\Q[a,b,c,d]$ that involves no $a^k$ with $k\geq 2$.

\subsection{Implementation of the variables}
\label{sec:ImplementationVariables}

Elements of $\Q_{a,b,c,d}$ are implemented as lists of vectors
$(\alpha,k_a,k_b,k_c,k_d)\in\Q\times \{0,1\}\times\N^3\subset \Q\times\N^4$, corresponding to the sum
of the $\alpha a^{k_a}b^{k_b}c^{k_c}d^{k_d}$. Addition and
multiplication in $\Q_{a,b,c,d}$ are implemented accordingly, using
the relation $a^2=1+ab+cd-b^2-c^2-d^2$ to remove terms with powers of
$a$ higher than 2.

Generators of
$\fract{\A_2\big(\BlModD\big)}{\A^{(2)}_2\big(\BlModD\big)}$ are
separated between 6--legs and 4--legs ones. The former are implemented
as
$\big((k_1,i_1),\ldots,(k_6,i_6)\big)\in\big(\Z\times\{1,2\}\big)^6$
corresponding to
\[
\sixlegsGR{t^{k_1}\gamma_{i_1}}{t^{k_2}\gamma_{i_2}}{t^{k_3}\gamma_{i_3}}{t^{k_4}\gamma_{i_4}}{t^{k_5}\gamma_{i_5}}{t^{k_6}\gamma_{i_6}}
\]
and the latter as $\big((k_1,i_1),\ldots,(k_4,i_4)\big)\in\big(\Z\times\{1,2\}\big)^4$
corresponding to
\[
\fourlegsGR{t^{k_1}\gamma_{i_1}}{t^{k_2}\gamma_{i_2}}{t^{k_3}\gamma_{i_3}}{t^{k_4}\gamma_{i_4}}.
\]
In both cases, the linking between legs $v$ and $w$ labelled by $t^{k_j}\gamma_{i_j}$ and $t^{k_\ell}\gamma_{i_\ell}$ is $f_{vw}=t^{k_j-k_\ell}\frac r\delta$. 
General elements of $\fract{\A_2\big(\BlModD\big)}{\A^{(2)}_2\big(\BlModD\big)}$ are implemented in two ways:
\begin{itemize}
\item for inputs: as linear combinations of the above generators;
\item for outputs: as vectors $(\alpha_1,\alpha_2,\alpha_3,\alpha_4,\alpha_5,\alpha_6)\in\Q_{a,b,c,d}^6$ corresponding
to the linear combination $\alpha_1\Gamma_1+\alpha_2\Gamma_2+\alpha_3H_1+\alpha_4H_2+\alpha_5H_3+\alpha_6H_4$, where the $H_i$ and the $\Gamma_i$ are given in Figures \ref{fig:Haches} and \ref{figGamma1}.
\end{itemize}

\subsection{Reduction algorithms}

The programs are based on two reduction algorithms \texttt{reduc4} and \texttt{reduc6}, one for 4--legs
generators and one for 6--legs generators. Both algorithms take, as input, a diagram $\Gamma$ implemented as
an element of $\big(\Z\times\{1,2\}\big)^{4\textrm{ or }6}$ representing one of the above generators and send, as output, a vector $(\alpha_1,\ldots,\alpha_6)\in\Q_{a,b,c,d}^6$
which expresses $\Gamma$ as $\Gamma=\alpha_1\Gamma_1+\alpha_2\Gamma_2+\alpha_3H_1+\alpha_4H_2+\alpha_5H_3+\alpha_6H_4$.

\vspace{2ex}

The \texttt{reduc4} algorithm goes as follows.
\begin{description}[topsep=.5ex,itemsep=-.5ex]
\item[{\tt Take}] $\big((k_1,e_1),(k_2,e_2),(k_3,e_3),(k_4,e_4)\big)$.
  (Call it $\Gamma$.)
\item[{\tt Check if}]
  \begin{description}[topsep=.5ex,itemsep=-.5ex]
  \item[$e_1+e_2+e_3+e_4$ {\tt is odd}] (that is if one of the $\Al_i$
    appears an odd number of times),
    \item[{\tt or if}] $(k_1,e_1)=(k_2,e_2)$ {\tt or} $(k_3,e_3)=(k_4,e_4)$
    (that is if two legs adjacent to a same
    trivalent vertex share the same label);
    \item[{\tt if so then send}] $(0,0,0,0,0,0)$.
  \end{description}
\item[$\longrightarrow$] At this point, legs
  sharing an adjacent trivalent vertex have distinct labels, and each $\Al_i$ appears 0, 2 or 4 times in leg labels.
\item[{\tt Check if}]
  \begin{description}[topsep=.5ex,itemsep=-.5ex]
  \item[{\tt some $k_i$ is $<0$ or $>1$};]
    \item[{\tt if so then}] {\tt send the sum of the results of reduc4 applied to
      the elements\\given by Corollary \ref{cor:DV} to increase or
      decrease $k_i$.}
  \end{description}
\item[$\longrightarrow$] At this point, each leg label is either some $\gamma_i$ or some $t\gamma_i$.
\item[{\tt Check if}]
  \begin{description}[topsep=.5ex,itemsep=-.5ex]
  \item[$e_1=e_2=e_3=e_4$] (that is if all legs are labelled in the same
    $\Al_i$; if so then $\Gamma$ is either
    $\fourlegsop{\gamma_i}{t\gamma_i}{\gamma_i}{t\gamma_i}{.1}$,
    $\fourlegsop{t\gamma_i}{\gamma_i}{t\gamma_i}{\gamma_i}{.1}$,
    $\fourlegsop{\gamma_i}{t\gamma_i}{t\gamma_i}{\gamma_i}{.1}$
    or $\fourlegsop{t\gamma_i}{\gamma_i}{\gamma_i}{t\gamma_i}{.1}$);
    \item[{\tt if so then send}]\ \\$(-1)^{k_1+k_3}\Big(${\tt
      reduc4}$\big((0,1),(1,1),(0,2),(1,2)\big)${\tt
      +reduc4}$\big((0,1),(1,2),(0,1),(1,2)\big)$\\{\tt
      +reduc4}$\big((0,1),(1,2),(0,2),(1,1)\big)\Big)$ (see \cite[Proposition 7.10]{M7}).
  \end{description}
\item[$\longrightarrow$] At this point, each $\Al_i$ appears exactly twice in leg labels.
\item[{\tt Check if}]
  \begin{description}[topsep=.5ex,itemsep=-.5ex]
  \item[$e_1=e_2$] (that is if the two $\Al_1$--labelled legs are both on the
    left or both on the right), 
  \item[{\tt if so then send}]\ \\{\tt
      reduc4}$\big((k_1,e_1),(k_3,e_3),(k_2,e_2),(k_4,e_4)\big)${\tt
      -reduc4}$\big((k_1,e_1),(k_4,e_4),(k_2,e_2),(k_3,e_3)\big)$ \\(using an $\IHX$
    move).
  \end{description}
\item[{\tt Check if}]
  \begin{description}[topsep=.5ex,itemsep=-.5ex]
  \item[$e_1=e_4$] (that is if the two $\Al_1$--labelled legs are both at the
    top or both at the bottom), 
  \item[{\tt if so then send}]{\tt
      -reduc4}$\big((k_1,e_1),(k_2,e_2),(k_4,e_4),(k_3,e_3)\big)$ (using an $\AS$
    move).
  \end{description}
\item[$\longrightarrow$] At this point, each $\Al_i$ appears simultaneously in labels of
  opposite legs only.
\item[{\tt Use}] $S:=k_1+k_2+k_3+k_4$ {\tt and, if $S=2$, the parity
    of $k_1+k_2$ and $k_1+k_2$ to determine to which element, among $H_1$,
    $H_2$, $H_3$ or $H_4$, $\Gamma$ is equal to, and send the\\corresponding output.}
 \end{description}

\vspace{2ex}

The \texttt{reduc6} algorithm goes as follows.
\begin{description}[topsep=.5ex,itemsep=-.5ex]
\item[{\tt Take}] $\big((k_1,e_1),(k_2,e_2),(k_3,e_3),(k_4,e_4),(k_5,e_5),(k_6,e_6)\big)$.
  (Call it $\Gamma$.)
\item[{\tt Check if}]
  \begin{description}[topsep=.5ex,itemsep=-.5ex]
  \item[$e_1+e_2+e_3+e_4+e_5+e_6$ {\tt is odd}] (that is if one of the $\Al_i$
    appears an odd number of times),
    \item[{\tt or if}] $(k_1,e_1)=(k_2,e_2)$ {\tt or}
      $(k_2,e_2)=(k_3,e_3)$ {\tt or} $(k_3,e_3)=(k_1,e_1)$ {\tt or}
      $(k_4,e_4)=(k_5,e_5)$ {\tt or} $(k_5,e_5)=(k_6,e_6)$ {\tt or} $(k_6,e_6)=(k_4,e_4)$
    (that is if two legs adjacent to a same
    trivalent vertex share the same label);
    \item[{\tt if so then send}] $(0,0,0,0,0,0)$.
  \end{description}
\item[$\longrightarrow$] At this point, legs
  sharing an adjacent trivalent vertex have distinct labels, and each
  $\Al_i$ appears an even number of times in leg labels.
\item[{\tt Check if}]
  \begin{description}[topsep=.5ex,itemsep=-.5ex]
  \item[{\tt some $k_i$ is $<0$ or $>1$};]
    \item[{\tt if so then}] {\tt send the sum of the results of reduc6
        and reduc4 applied to
      the\\elements given by Corollary \ref{cor:DV} to increase or
      decrease $k_i$.}
  \end{description}
\item[$\longrightarrow$] At this point, each leg label is either some $\gamma_i$ or some
  $t\gamma_i$, and each $\Al_i$ appears 2 or 4 times in leg
  labels---if all legs were $\Al_i$--labelled, then two legs sharing a
  same adjacent trivalent vertex would have a same label.
\item[{\tt Check if}]
  \begin{description}[topsep=.5ex,itemsep=-.5ex]
  \item[$e_1+e_2+e_3+e_4+e_5+e_6=8$] (that is if $\Al_1$ appears 4
    times and $\Al_2$ twice in leg labels),
    \item[{\tt if so then}]{\tt send}\\
        {\tt reduc6}\hbox{$\big((k_1,3-e_1),(k_2,3-e_2),(k_3,3-e_3),(k_4,3-e_4),(k_5,3-e_5),(k_6,3-e_6)\big)$}\\
      (using a $\Aut_\xi$ move).
  \end{description}
\item[$\longrightarrow$] At this point, $\Al_1$ appears twice and
  $\Al_2$ four times in leg labels, and the two $\Al_1$--labelled legs
  are on distinct connected components of $\gamma$, otherwise two
  $\Al_2$--labelled legs sharing a
  same adjacent trivalent vertex would have a same label.
\item[{\tt Check if}]
  \begin{description}[topsep=.5ex,itemsep=-.5ex]
  \item[$e_i=1$] for $i\in\{2,3,5,6\}$ (that is if the two
    $\Al_1$--labelled legs are not both at the top), 
  \item[{\tt if so then send}] {\tt
      reduc6}$\big((k'_1,e'_1),(k'_2,e'_2),(k'_3,e'_3),(k'_4,e'_4),(k'_5,e'_5),(k'_6,e'_6)\big)$
    {\tt where\\
      $\big((k'_1,e'_1),(k'_2,e'_2),(k'_3,e'_3)\big)$
    and $\big((k'_4,e'_4),(k'_5,e'_5),(k'_6,e'_6)\big)$ are
    respectively the\\cyclic permutations of $\big((k_1,e_1),(k_2,e_2),(k_3,e_3)\big)$
    and $\big((k_4,e_4),(k_5,e_5),(k_6,e_6)\big)$\\such that $e'_1=e'_4=1$.}
  \end{description}
\item[$\longrightarrow$] At this point, the two legs at the top are
  $\Al_1$--labelled and the four other are $\Al_2$--labelled, with, on each connected component of $\Gamma$, one occurence of $\gamma_2$ and one occurence 
  of $t\gamma_2$.
\item[{\tt Use}] $k_3+k_5-k_2-k_6$ {\tt and the parity
    of $k_1+k_4$ to determine to which element, among $\pm\Gamma_1$ or
    $\pm\Gamma_2$, $\Gamma$ is equal to, and send the corresponding output.}
 \end{description}

\subsection{Computations and results}

\begin{figure}[t]
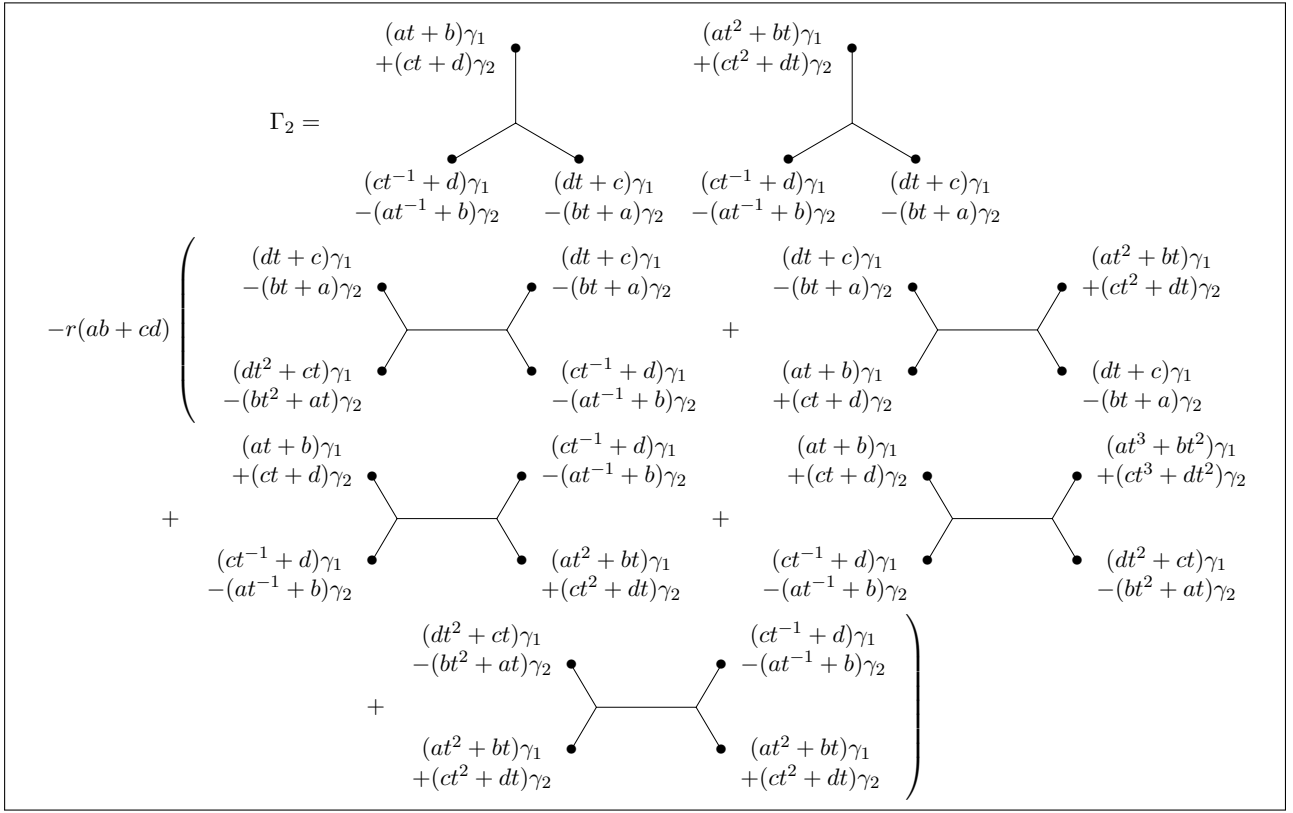

\CenterFormula{\fbox{\scalebox{0.82}{$
\begin{array}{c}
  \Gamma_2=\sixlegsGRR
  {\begin{array}{c}(at+b)\gamma_1\\+(ct+d)\gamma_2\end{array}}
  {\begin{array}{c}(ct^{-1}+d)\gamma_1\\-(at^{-1}+b)\gamma_2\end{array}}
  {\begin{array}{c}(dt+c)\gamma_1\\-(bt+a)\gamma_2\end{array}}
  {\begin{array}{c}(at^2+bt)\gamma_1\\+(ct^2+dt)\gamma_2\end{array}}
  {\begin{array}{c}(ct^{-1}+d)\gamma_1\\-(at^{-1}+b)\gamma_2\end{array}}
  {\begin{array}{c}(dt+c)\gamma_1\\-(bt+a)\gamma_2\end{array}}\\
-r(ab+cd)\left(\fourlegsGR
{\begin{array}{c}(dt+c)\gamma_1\\-(bt+a)\gamma_2\\\strut\end{array}}
{\begin{array}{c}\strut\\(dt^2+ct)\gamma_1\\-(bt^2+at)\gamma_2\end{array}}
{\begin{array}{c}\strut\\
    (ct^{-1}+d)\gamma_1\\-(at^{-1}+b)\gamma_2\end{array}}
{\begin{array}{c}(dt+c)\gamma_1\\-(bt+a)\gamma_2\\\strut\end{array}}
+\fourlegsGR
{\begin{array}{c}(dt+c)\gamma_1\\-(bt+a)\gamma_2\\\strut\end{array}}
{\begin{array}{c}\strut\\(at+b)\gamma_1\\+(ct+d)\gamma_2\end{array}}
{\begin{array}{c}\strut\\ (dt+c)\gamma_1\\-(bt+a)\gamma_2\end{array}}
{\begin{array}{c}(at^2+bt)\gamma_1\\+(ct^2+dt)\gamma_2\\\strut\end{array}}\right.\\
\hspace{2.2cm}+\fourlegsGR
{\begin{array}{c}(at+b)\gamma_1\\+(ct+d)\gamma_2\\\strut\end{array}}
{\begin{array}{c}\strut\\(ct^{-1}+d)\gamma_1\\-(at^{-1}+b)\gamma_2\end{array}}
{\begin{array}{c}\strut\\
    (at^2+bt)\gamma_1\\+(ct^2+dt)\gamma_2\end{array}}
{\begin{array}{c}(ct^{-1}+d)\gamma_1\\-(at^{-1}+b)\gamma_2\\\strut\end{array}}
+\fourlegsGR
{\begin{array}{c}(at+b)\gamma_1\\+(ct+d)\gamma_2\\\strut\end{array}}
{\begin{array}{c}\strut\\(ct^{-1}+d)\gamma_1\\-(at^{-1}+b)\gamma_2\end{array}}
{\begin{array}{c}\strut\\
    (dt^2+ct)\gamma_1\\-(bt^2+at)\gamma_2\end{array}}
{\begin{array}{c}(at^3+bt^2)\gamma_1\\+(ct^3+dt^2)\gamma_2\\\strut\end{array}}\\
\left.+\fourlegsGR
{\begin{array}{c}(dt^2+ct)\gamma_1\\-(bt^2+at)\gamma_2\\\strut\end{array}}
{\begin{array}{c}\strut\\(at^2+bt)\gamma_1\\+(ct^2+dt)\gamma_2\end{array}}
{\begin{array}{c}\strut\\
    (at^2+bt)\gamma_1\\+(ct^2+dt)\gamma_2\end{array}}
{\begin{array}{c}(ct^{-1}+d)\gamma_1\\-(at^{-1}+b)\gamma_2\\\strut\end{array}}\right)
\end{array}$}}}
\caption{Input for $\lambda_{a,b,c,d}.\Gamma_2$} \label{figGamma2Input}
\end{figure}

\begin{figure}[t]
\CenterFormula{\fbox{\scalebox{0.82}{$
\begin{array}{c}
  \Gamma_1=\sixlegsGRR
  {\begin{array}{c}(at+b)\gamma_1\\+(ct+d)\gamma_2\end{array}}
  {\begin{array}{c}(ct^{-1}+d)\gamma_1\\-(at^{-1}+b)\gamma_2\end{array}}
  {\begin{array}{c}(dt+c)\gamma_1\\-(bt+a)\gamma_2\end{array}}
  {\begin{array}{c}(at+b)\gamma_1\\+(ct+d)\gamma_2\end{array}}
  {\begin{array}{c}(ct^{-1}+d)\gamma_1\\-(at^{-1}+b)\gamma_2\end{array}}
  {\begin{array}{c}(dt+c)\gamma_1\\-(bt+a)\gamma_2\end{array}}\\
-r(ab+cd)\left(\fourlegsGR
{\begin{array}{c}(ct^{-1}+d)\gamma_1\\-(at^{-1}+b)\gamma_2\\\strut\end{array}}
{\begin{array}{c}\strut\\(dt+c)\gamma_1\\-(bt+a)\gamma_2\end{array}}
{\begin{array}{c}\strut\\
    (ct^{-1}+d)\gamma_1\\-(at^{-1}+b)\gamma_2\end{array}}
{\begin{array}{c}(dt+c)\gamma_1\\-(bt+a)\gamma_2\\\strut\end{array}}
+\fourlegsGR
{\begin{array}{c}(dt+c)\gamma_1\\-(bt+a)\gamma_2\\\strut\end{array}}
{\begin{array}{c}\strut\\(at+b)\gamma_1\\+(ct+d)\gamma_2\end{array}}
{\begin{array}{c}\strut\\ (dt+c)\gamma_1\\-(bt+a)\gamma_2\end{array}}
{\begin{array}{c}(at+b)\gamma_1\\+(ct+d)\gamma_2\\\strut\end{array}}\right.\\
\hspace{2.2cm}+\fourlegsGR
{\begin{array}{c}(at+b)\gamma_1\\+(ct+d)\gamma_2\\\strut\end{array}}
{\begin{array}{c}\strut\\(ct^{-1}+d)\gamma_1\\-(at^{-1}+b)\gamma_2\end{array}}
{\begin{array}{c}\strut\\
    (at+b)\gamma_1\\+(ct+d)\gamma_2\end{array}}
{\begin{array}{c}(ct^{-1}+d)\gamma_1\\-(at^{-1}+b)\gamma_2\\\strut\end{array}}
+\fourlegsGR
{\begin{array}{c}(at+b)\gamma_1\\+(ct+d)\gamma_2\\\strut\end{array}}
{\begin{array}{c}\strut\\(ct^{-1}+d)\gamma_1\\-(at^{-1}+b)\gamma_2\end{array}}
{\begin{array}{c}\strut\\
    (dt^2+ct)\gamma_1\\-(bt^2+at)\gamma_2\end{array}}
{\begin{array}{c}(at^2+bt)\gamma_1\\+(ct^2+dt)\gamma_2\\\strut\end{array}}\\
\left.+\fourlegsGR
{\begin{array}{c}(dt^2+ct)\gamma_1\\-(bt^2+at)\gamma_2\\\strut\end{array}}
{\begin{array}{c}\strut\\(at^2+bt)\gamma_1\\+(ct^2+dt)\gamma_2\end{array}}
{\begin{array}{c}\strut\\
    (at+b)\gamma_1\\+(ct+d)\gamma_2\end{array}}
{\begin{array}{c}(ct^{-1}+d)\gamma_1\\-(at^{-1}+b)\gamma_2\\\strut\end{array}}\right)
\end{array}$}}}
\caption{Input for $\lambda_{a,b,c,d}.\Gamma_1$} \label{figGamma1Input}
\end{figure}

\begin{figure}[t]
\CenterFormula{\fbox{\scalebox{0.82}{$
\begin{array}{c}
  \Gamma_1=\sixlegsGRR
  {\begin{array}{c}(at^2+bt)\gamma_1\\+(ct^2+dt)\gamma_2\end{array}}
  {\begin{array}{c}(ct^{-1}+d)\gamma_1\\-(at^{-1}+b)\gamma_2\end{array}}
  {\begin{array}{c}(dt+c)\gamma_1\\-(bt+a)\gamma_2\end{array}}
  {\begin{array}{c}(at^2+bt)\gamma_1\\+(ct^2+dt)\gamma_2\end{array}}
  {\begin{array}{c}(ct^{-1}+d)\gamma_1\\-(at^{-1}+b)\gamma_2\end{array}}
  {\begin{array}{c}(dt+c)\gamma_1\\-(bt+a)\gamma_2\end{array}}\\
-r(ab+cd)\left(\fourlegsGR
{\begin{array}{c}(ct^{-1}+d)\gamma_1\\-(at^{-1}+b)\gamma_2\\\strut\end{array}}
{\begin{array}{c}\strut\\(dt+c)\gamma_1\\-(bt+a)\gamma_2\end{array}}
{\begin{array}{c}\strut\\
    (ct^{-1}+d)\gamma_1\\-(at^{-1}+b)\gamma_2\end{array}}
{\begin{array}{c}(dt+c)\gamma_1\\-(bt+a)\gamma_2\\\strut\end{array}}
+\fourlegsGR
{\begin{array}{c}(dt+c)\gamma_1\\-(bt+a)\gamma_2\\\strut\end{array}}
{\begin{array}{c}\strut\\(at^2+bt)\gamma_1\\+(ct^2+dt)\gamma_2\end{array}}
{\begin{array}{c}\strut\\ (dt+c)\gamma_1\\-(bt+a)\gamma_2\end{array}}
{\begin{array}{c}(at^2+bt)\gamma_1\\+(ct^2+dt)\gamma_2\\\strut\end{array}}\right.\\
\hspace{2.2cm}+\fourlegsGR
{\begin{array}{c}(at^2+bt)\gamma_1\\+(ct^2+dt)\gamma_2\\\strut\end{array}}
{\begin{array}{c}\strut\\(ct^{-1}+d)\gamma_1\\-(at^{-1}+b)\gamma_2\end{array}}
{\begin{array}{c}\strut\\
    (at^2+bt)\gamma_1\\+(ct^2+dt)\gamma_2\end{array}}
{\begin{array}{c}(ct^{-1}+d)\gamma_1\\-(at^{-1}+b)\gamma_2\\\strut\end{array}}
+\fourlegsGR
{\begin{array}{c}(at^2+bt)\gamma_1\\+(ct^2+dt)\gamma_2\\\strut\end{array}}
{\begin{array}{c}\strut\\(ct^{-1}+d)\gamma_1\\-(at^{-1}+b)\gamma_2\end{array}}
{\begin{array}{c}\strut\\
    (dt^2+ct)\gamma_1\\-(bt^2+at)\gamma_2\end{array}}
{\begin{array}{c}(at^3+bt^2)\gamma_1\\+(ct^3+dt^2)\gamma_2\\\strut\end{array}}\\
\left.+\fourlegsGR
{\begin{array}{c}(dt^2+ct)\gamma_1\\-(bt^2+at)\gamma_2\\\strut\end{array}}
{\begin{array}{c}\strut\\(at^3+bt^2)\gamma_1\\+(ct^3+dt^2)\gamma_2\end{array}}
{\begin{array}{c}\strut\\
    (at^2+bt)\gamma_1\\+(ct^2+dt)\gamma_2\end{array}}
{\begin{array}{c}(ct^{-1}+d)\gamma_1\\-(at^{-1}+b)\gamma_2\\\strut\end{array}}\right)
\end{array}$}}}
\caption{Input for $\lambda_{a,b,c,d}.\Gamma_3$} \label{figGamma3Input}
\end{figure}

As the computation for $\Gamma_2$ is slightly more complicated than for $\Gamma_1$ and $\Gamma_3$, we start with $\Gamma_2$. 
The action of $\lambda_{a,b,c,d}$ on $\Gamma_2$ produces:
\[
\sixlegsGRR
  {\begin{array}{c}(at+b)\gamma_1\\+(ct+d)\gamma_2\end{array}}
  {\begin{array}{c}(ct^{-1}+d)\gamma_1\\-(at^{-1}+b)\gamma_2\end{array}}
  {\begin{array}{c}(dt+c)\gamma_1\\-(bt+a)\gamma_2\end{array}}
  {\begin{array}{c}(at^2+bt)\gamma_1\\+(ct^2+dt)\gamma_2\end{array}}
  {\begin{array}{c}(ct^{-1}+d)\gamma_1\\-(at^{-1}+b)\gamma_2\end{array}}
  {\begin{array}{c}(dt+c)\gamma_1\\-(bt+a)\gamma_2\end{array}},
\]
with the same linkings as in $\Gamma_2$. However, in our implementation of the diagrams as linear combinations of the generators described 
in Section \ref{sec:ImplementationVariables}, the convention gives, for two legs $v$ and $w$ 
labelled by $P\gamma_1+Q\gamma_2$ and $R\gamma_1+S\gamma_2$ respectively, a linking equal to $f_{vw}=(P\bar{R}+Q\bar{S})\frac r\delta$. 
For instance, numbering the vertices as $\sixlegsSchema$, we have $f^{\lambda_{a,b,c,d}.\Gamma_2}_{14}=f^{\Gamma_2}_{14}=\frac{rt^{-1}}\delta$ 
whereas the linking in the above diagram is
\[
\begin{array}{rcl}
f_{14}&=&\frac{r\big((at+b)(at^{-2}+bt^{-1})+(ct+d)(ct^{-2}+dt^{-1})\big)}\delta\\
&=&\frac{r\big((ab+cd)+(a^2+b^2+c^2+d^2)t^{-1}+(ab+cd)t^{-2}\big)}\delta\\
&=&\frac{r\big((a^2+b^2+c^2+d^2-ab-cd)t^{-1}+(ab+cd)t^{-1}\delta\big)}\delta\\
& =& \frac{rt^{-1}}\delta+r(ab+cd)t^{-1}.
\end{array}
\]

\noindent This can be fixed, thanks to $\LV$, by adding a term
\[
-r(ab+cd)\left(\fourlegsGR
{\begin{array}{c}(dt+c)\gamma_1\\-(bt+a)\gamma_2\\\strut\end{array}}
{\begin{array}{c}\strut\\(dt^2+ct)\gamma_1\\-(bt^2+at)\gamma_2\end{array}}
{\begin{array}{c}\strut\\
    (ct^{-1}+d)\gamma_1\\-(at^{-1}+b)\gamma_2\end{array}}
{\begin{array}{c}(dt+c)\gamma_1\\-(bt+a)\gamma_2\\\strut\end{array}}\right).
\]
Likewise, the linking $f^{\Gamma_2}_{25}$, $f^{\Gamma_2}_{36}$,
$f^{\Gamma_2}_{35}$ and $f^{\Gamma_2}_{26}$ can be fixed by adding
similar 4--legs terms. All the other linkings vanish already as expected.
Finally, we get the decomposition of $\Gamma_2$ given in Figure \ref{figGamma2Input}.

To compute the corresponding relation, we defined six matrices, one
for each term in the formula of Figure \ref{figGamma2Input}, rows corresponding to legs and columns to each of the four monomials that appear in the leg
labels. The program uses these matrices to develop with $\LV$
the six diagrams in order to get a weighted sum of generators, as they are described in Section~\ref{sec:ImplementationVariables}. 
Then, by applying either {\tt reduc4} or {\tt reduc6} to each term in this weighted sum, it expresses it as a linear combination of $\Gamma_1$, 
$\Gamma_2$ and the $H_i$'s. Finally, the program uses the relations 
$H_1=-2H_2$ and $H_4=-H_2-H_3$ from Lemma \ref{lemma4step4legs}---which hold in $\fract{\A_2\big(\BlModD\big)}{\A^{(2)}_2\big(\BlModD\big)}$ by 
the same computations as in $\fract{\A_2\big(\BlModT\big)}{\A^{(2)}_2\big(\BlModT\big)}$---to reduce this linear combination in terms of
$\Gamma_1$, $\Gamma_2$, $H_2$ and $H_3$ only. We end up with
\[
\Gamma_2=(b^2+d^2-ab-cd-1)\Gamma_1+(2b^2+2d^2-2ab-2cd-1)\Gamma_2+r(3ab+3cd-3b^2-3d^2+3)H_3,
\]
that is
\[
(a^2+c^2)(\Gamma_1+2\Gamma_2-3rH_3)=0.
\]
But it was already known that $\Gamma_1+2\Gamma_2=r\fourlegsU{}t{}t$ and the same computation as in the proof of Proposition \ref{proponecopy} gives
$\fourlegsU{}t{}t=H_1+H_3-2H_4=3H_3$.

Similarly, the action of $\lambda_{a,b,c,d}$ on $\Gamma_1$ leads to the decomposition given in Figure \ref{figGamma1Input}.
The program reduces it to
\[
\Gamma_1=(ab+cd+1)\Gamma_1+2(ab+cd)\Gamma_2-3r(ab+cd)H_3,
\]
that is 
\[
(ab+cd)(\Gamma_1+2\Gamma_2-3rH_3)=0,
\]
which recovers once again a previously known formula.

Finally, the action of $\lambda_{a,b,c,d}$ on $\Gamma_3$ leads to the decomposition given in Figure \ref{figGamma3Input}.
The program reduces it to
\[
\Gamma_1=(2-b^2-d^2)\Gamma_1+2(1-b^2-d^2)\Gamma_2+3r(b^2+d^2-1)H_3,
\]
that is 
\[
(b^2+d^2-1)(\Gamma_1+2\Gamma_2-3rH_3)=0,
\]
which still recovers the same previously known formula.

\end{appendix}

\def\cprime{$'$}
\providecommand{\bysame}{\leavevmode ---\ }
\providecommand{\og}{``}
\providecommand{\fg}{''}
\providecommand{\smfandname}{\&}
\providecommand{\smfedsname}{\'eds.}
\providecommand{\smfedname}{\'ed.}
\providecommand{\smfmastersthesisname}{M\'emoire}
\providecommand{\smfphdthesisname}{Th\`ese}

\end{document}